\documentclass{article}
\usepackage{amsmath,amsthm,amssymb,array}
\usepackage[active]{srcltx}

\theoremstyle{plain}  

\newtheorem{theorem}{Theorem}[section]

\newtheorem{proposition}[theorem]{Proposition}
\newtheorem{lemma}[theorem]{Lemma}

\usepackage{amsfonts}
\usepackage{amsmath}
\usepackage{amssymb}
\usepackage[bookmarksnumbered,plainpages]{hyperref}
\usepackage{graphicx}
\usepackage{fancyhdr}
\providecommand{\U}[1]{\protect\rule{.1in}{.1in}}
\providecommand{\U}[1]{\protect\rule{.1in}{.1in}}
\providecommand{\U}[1]{\protect\rule{.1in}{.1in}}
\providecommand{\U}[1]{\protect\rule{.1in}{.1in}}
\providecommand{\U}[1]{\protect\rule{.1in}{.1in}}
\input{xy}
\xyoption{all}
\newtheorem{remark}[theorem]{Remark}

\theoremstyle{definition}

\theoremstyle{remark}

\newcommand{\R}{\mathbb{R}}

\newcommand{\be}{\begin{equation}}	
\newcommand{\ee}{\end{equation}}	
\newcommand{\bern}{\begin{eqnarray}}	
\newcommand{\eern}{\end{eqnarray}}	

\newcommand{\norm}[1]{\left\|#1\right\|}		

\begin{document}

\title{\sc Concentration of solutions for a \\singularly perturbed mixed problem \\in non smooth domains}

\date{}

\maketitle

\begin{center}
\author{Serena Dipierro}
\end{center}

\begin{center}
{\small SISSA, Sector of Mathematical Analysis\\
Via Bonomea 265, 34136 Trieste, Italy\\
E-mail address: dipierro@sissa.it
}
\end{center}

\begin{abstract}
We consider a singularly perturbed problem with mixed Dirichlet and Neumann boundary conditions 
in a bounded domain $\Omega\subset\R^{n}$ whose boundary has an $\left(n-2\right)$-dimensional singularity. 
Assuming $1<p<\frac{n+2}{n-2}$, we prove that, under suitable geometric conditions on the boundary of the domain, 
there exist solutions which approach the intersection of the Neumann and the Dirichlet parts as 
the singular perturbation parameter tends to zero. 
\end{abstract}



\section{Introduction}

In this paper we study the following singular perturbation problem with mixed Dirichlet and Neumann boundary conditions 
in a bounded domain $\Omega\subset\R^{n}$ whose boundary $\partial\Omega$ is non smooth:
\bern
\left\{ 
\begin{array}{lll} 
         -\epsilon^ {2}\Delta u + u = u^ {p}   \quad     &   \mathrm{in\ } \Omega ,      \\
          \frac{\partial u}{\partial \nu} = 0  \quad \mathrm{on\ } \partial_{N}\Omega ,  & u=0 \quad  \mathrm{on\ } \partial_{D}\Omega ,  \\ 
          u>0   \quad  &     \mathrm{in\ } \Omega .
 \end{array} 
\right.         
 \label{problem}\eern 
Here $p\in\left(1, \frac{n+2}{n-2}\right)$ is subcritical, $\nu$ denotes the outer unit normal at $\partial\Omega$ and $\epsilon >0$ is a small parameter. 
Moreover $\partial_{N}\Omega$, $\partial_{D}\Omega$ 
are two subsets of the boundary of $\Omega$ such that the union of their closures coincides with the whole $\partial\Omega$, and their intersection is 
an $\left(n-2\right)$-dimensional smooth singularity. 

Problem $(\ref{problem})$ or some of its variants arise in several physical or 
biological models. 
Consider for instance the study of the \emph{population dynamics}: 
suppose that a species lives in a bounded region $\Omega$ whose boundary has 
two parts, $\partial_{N}\Omega$, which is an obstacle that blocks 
the pass across, and $\partial_{D}\Omega$, 
which is a killing zone for the population. 
Moreover $(\ref{problem})$ is a model of the \emph{heat conduction} 
for small conductivity, 
when there is a nonlinear source in the interior of the domain, with combined 
isothermal and isolated regions at the boundary. 

Concerning \emph{reaction-diffusion systems}, this phenomenon is 
related to the so-called Turing's instability. 
More precisely, for single equation with Neumann boundary conditions 
it is known that scalar reaction-diffusion equations 
in a convex domain admit only constant stable steady state solutions; 
see \cite{CH}, \cite{Mat}. 
On the other hand, as noticed in \cite{Tu}, 
reaction-diffusion systems with different 
diffusivities might generate 
non-homogeneous stable steady states. 
A well-known example is the Gierer-Meinhardt system, 
introduced in \cite{GM} to describe some biological experiment.  
We refer to \cite{Ni}, \cite{NTY} for more details. 

Another motivation comes from the Nonlinear Schr\"{o}dinger Equation
\be     i\hbar\frac{\partial\psi}{\partial t}   =    - \frac{\hbar^{2}}{2m} \Delta\psi + V\psi -  \gamma\vert\psi\vert^{p-2}\psi,   \nonumber\ee
where $\hbar$ is the Planck constant, $V$ is the potential, and $\gamma$ and $m$ 
are positive constants. 
In fact, if we analyze standing waves and consider the semiclassical limit 
$\hbar\rightarrow 0$, we obtain a singularly perturbed equation; 
see for example \cite{ABC}, \cite{AM}, \cite{AMS}, \cite{FW}, 
and references therein.   
 
\bigskip

Let us now describe some results which concern singularly perturbed problems 
with Neumann or Dirichlet boundary conditions, and specifically 
\bern   
\left\{ 
\begin{array}{lll} 
         -\epsilon^ {2}\Delta u + u = u^ {p}   \quad     & \mathrm{in\ } \Omega , \\
          \frac{\partial u}{\partial \nu} = 0\quad &\mathrm{on\ } \partial\Omega ,\\ 
          u>0   \quad  &     \mathrm{in\ } \Omega ,
 \end{array} 
\right.   
 \label{Neumann}\eern
and
\bern
\left\{ 
\begin{array}{lll} 
         -\epsilon^ {2}\Delta u + u = u^ {p}   \quad     &   \mathrm{in\ } \Omega , \\
           u=0 \quad    & \mathrm{on\ } \partial\Omega ,  \\ 
          u>0   \quad  &     \mathrm{in\ } \Omega .
 \end{array} 
\right.            
 \label{Dirichlet}\eern 

The study of the concentration phenomena at points 
for smooth domains is very rich and has been intensively developed 
in recent years. 
The search for such condensing solutions 
is essentially carried out by two methods. 
The first approach is variational and uses tools 
of the critical point theory or topological methods. 
A second way is to reduce the problem to a finite-dimensional 
one by means of Lyapunov-Schmidt reduction.

The typical concentration behavior of solution $U_{Q, \epsilon}$ 
to $(\ref{problem})$ is via a scaling of the variables in the form 
\be    U_{Q,\epsilon}\left(x\right)\sim U\left(\frac{x-Q}{\epsilon}\right), \label{appr}\ee
where $Q$ is some point of $\bar{\Omega}$, 
and $U$ is a solution of the problem 
\be       -\Delta U + U = U^ {p}   \quad  \mathrm{in\ } \R^{n}    
               \quad  \mathrm{(or\ in\ } \R^{n}_{+}=\left\lbrace \left(x_{1}, \ldots , x_{n}\right)\in\R^{n}  :  x_{n}>0\right\rbrace),
\label{prob}\ee   
the domain depending on whether $Q$ lies in the interior of $\Omega$ 
or at the boundary. 
When $p<\frac{n+2}{n-2}$ (and indeed only if this inequality is satisfied), 
problem $(\ref{prob})$ admits positive radial solutions 
which decay to zero at infinity; see \cite{BL}, \cite{St}. 
Solutions of $(\ref{problem})$ with this profile 
are called \textit{spike-layers}, 
since they are highly concentrated near some point of $\bar{\Omega}$. 

Consider first the problem with pure Neumann boundary conditions. 
Solutions of $(\ref{Neumann})$ 
with a concentration at one or more points of 
the boundary $\partial\Omega$ as $\epsilon\rightarrow 0$ are called 
\textit{boundary-spike layers}. 
They are peaked near critical points of the mean curvature. 
In particular, it was shown in \cite{NT1}, \cite{NT2} that 
mountain-pass solutions of $(\ref{Neumann})$ 
concentrate at $\partial\Omega$ near 
global maxima of the mean curvature.  
One can see this fact considering the variational structure of the problem. 
In fact, solutions of $(\ref{Neumann})$ can be found as critical points 
of the following Euler-Lagrange functional 
\be
 I_{\epsilon, N}\left(u\right)  = \frac{1}{2} \int_{\Omega}\left(\epsilon^ {2}\vert\nabla u\vert^ {2} + u^ {2}\right)dx - \frac{1}{p+1}\int_{\Omega}\vert u\vert^ {p+1}dx, \quad  u\in H^{1}\left(\Omega\right).
\nonumber\ee
Plugging into $I_{\epsilon, N}$ a function of the form 
$(\ref{appr})$ with $Q\in\partial\Omega$ 
one sees that 
\be   I_{\epsilon, N}\left(U_{Q, \epsilon}\right)=C_{0}\epsilon^{n}-C_{1}\epsilon^{n+1}H\left(Q\right)+o\left(\epsilon^{n+1}\right), \label{mean}\ee 
where $C_{0}, C_{1}$ are positive constants depending only 
on $n$ and $p$, and $H$ is the mean curvature; 
see for instance \cite{AM}, Lemma $9.7$. 
To obtain this expansion one can use the radial symmetry of $U$ 
and parametrize $\partial\Omega$ as a normal graph near $Q$. 
From the above formula one can see that the bigger 
is the mean curvature the lower is the energy of this function: 
roughly speaking, boundary spike layers would tend to move 
along the gradient of $H$ in order to minimize their energy. 
Moreover one can say that the energy of spike-layers is of order $\epsilon^{n}$, 
which is proportional to the volume of their \textit{support}, 
heuristically identified with a ball of radius 
$\epsilon$ centered at the peak. 
There is an extensive literature regarding the search of 
more general solutions of $(\ref{Neumann})$ 
concentrating at critical points of $H$; 
see \cite{DFW}, \cite{Gr}, \cite{GPW}, \cite{Gu}, \cite{Li}, \cite{LNT}, \cite{NPT}, \cite{We}.

Consider now the problem with pure Dirichlet boundary conditions. 
In this case spike layers with minimal energy concentrate at the interior of the domain, 
at points which maximize the distance from the boundary; 
see \cite{LN}, \cite{NW}. 
The intuitive reason for this is that, if $Q$ is in the interior of $\Omega$ 
and if we want to adapt a function like $(\ref{appr})$ to the 
Dirichlet conditions, the adjustment needs an energy which increases 
as $Q$ becomes closer and closer to $\partial\Omega$. 
Following the above heuristic argument, we could say that spike layers 
are \emph{repelled} from the regions where Dirichlet conditions are imposed.

Concerning mixed problem $(\ref{problem})$, in two recent papers \cite{GMMP1}, \cite{GMMP2} 
it was proved that, under suitable geometric conditions on 
the boundary of a smooth domain, 
there exist solutions which approach the intersection 
of the Neumann and the Dirichlet parts as the singular perturbation 
parameter tends to zero. 
In fact, denoting by $u_{\epsilon, Q}$ an approximate solution peaked at $Q$ 
and by $d_{\epsilon}$ the distance of $Q$ from the interface between 
$\partial_{N}\Omega$ and $\partial_{D}\Omega$, 
then its energy turns out to be the following 
\be I_{\epsilon}\left(u_{Q, \epsilon}\right)= C_{0}\epsilon^{n}   -  C_{1}\epsilon^{n+1}H\left(Q\right) + 
\epsilon^{n} e^{-2\frac{d_{\epsilon}}{\epsilon}\left(1 + o\left(1\right) \right)}  + 
o\left(\epsilon^{n+2}\right), \label{energy}\ee 
where $I_{\epsilon}$ is the functional associated to the mixed problem. 
Note that the first two terms in $(\ref{energy})$ are as in the 
expansion $(\ref{mean})$, while the third one represents a sort of 
\emph{potential energy} which decreases with the distance of $Q$ from the interface, 
consistently with the \emph{repulsive effect} which was described before 
for $(\ref{Dirichlet})$.

\bigskip

In almost all the papers mentioned above the case of $\Omega$ smooth was considered. 
Concerning instead the case of $\Omega$ non smooth, in \cite{Di} the author 
studied the concentration of solutions of the Neumann problem $(\ref{Neumann})$ at suitable points 
of the boundary of a non-smooth domain. 
Assuming for simplicity that $\Omega\subset\R^{3}$ is a 
piecewise smooth bounded domain whose boundary $\partial\Omega$ 
has a finite number of smooth edges, 
one can fix an edge $\Gamma$ on the boundary and consider 
the function $\alpha:\Gamma\rightarrow\R$ 
which associates to every $Q\in\Gamma$ 
the opening angle at $Q$, $\alpha\left(Q\right)$. 
Then it was proved that this function 
plays a similar role as the mean curvature $H$ for a smooth domain. 
In fact, plugging into $I_{\epsilon, N}$ a function of the form 
$(\ref{appr})$ with $Q\in\Gamma$, 
one obtains the analogous expression to $(\ref{mean})$ for 
this kind of domains, 
with $C_{0}\alpha\left(Q\right)$ instead of $C_{0}$. 
Again, one can give an heuristic explanation 
considering the fact that in this case one has to intersect 
the ball of radius $\epsilon$, which is identified with the support 
of the solution, with the domain, obtaining the dependence on the angle $\alpha\left(Q\right)$.

\bigskip

We are interested here in finding boundary spike-layers 
for the mixed problem $(\ref{problem})$. 
We call $\Gamma$ the intersection of the closures of $\partial_{N}\Omega$ and $\partial_{D}\Omega$, 
and suppose that it is an $(n-2)$-dimensional smooth singularity. 
Moreover we denote by $H$ the mean curvature of $\partial\Omega$ restricted to the closure of $\partial_{N}\Omega$, 
that is $H:\overline{\partial_{N}\Omega}\rightarrow\R$.

The main result of this paper is the following:
\begin{theorem} 
Let $\Omega\subset\R^{n}$, $n\geq 2$, be a bounded domain 
whose boundary $\partial\Omega$ has an $(n-2)$-dimensional smooth singularity, 
and $1<p<\frac{n+2}{n-2}$ ($1<p<+\infty$ if $n=2$). 
Suppose that $\partial_{N}\Omega$, $\partial_{D}\Omega$ are disjoint open sets of $\partial\Omega$ 
such that the union of the closures is the whole boundary of $\Omega$ and 
such that their intersection $\Gamma$ is the singularity. 
Suppose $Q\in\Gamma$ is such that $\alpha\left(Q\right)\neq 0$ and 
$H\vert_{\Gamma}$ 
is critical and non degenerate at $Q$, and that 
$\nabla H\left(Q\right)\neq 0$ points toward $\partial_{D}\Omega$. 
Then for $\epsilon >0$ sufficiently small problem 
$(\ref{problem})$ admits a solution concentrating at $Q$.
\label{th:solution}
\end{theorem}

\begin{remark} 
\begin{itemize}
\item[(a)] The non degeneracy condition in Theorem $\ref{th:solution}$ 
can be replaced by the condition that $Q$ is a strict local maximum or minimum of $H\vert_{\Gamma}$, 
or by the fact that there exists an open set $V$ of $\Gamma$ 
containing $Q$ such that $H\left(Q\right)>\sup_{\partial V}H$ 
or $H\left(Q\right)<\inf_{\partial V}H$.  
\item[(b)] With more precision, as $\epsilon\rightarrow 0$, the above solution 
possesses a unique global maximum point $Q_{\epsilon}\in\partial_{N}\Omega$, 
and $dist\left(Q_{\epsilon}, \Gamma\right)$ is of order 
$\epsilon\log\frac{1}{\epsilon}$. 
\end{itemize}
\label{rem:th}\end{remark}

The general strategy for proving Theorem $\ref{th:solution}$ 
relies on a finite-dimensional reduction; 
see for example the book \cite{AM}. 
One finds first a manifold $Z$ of approximate solutions to the given problem, 
which in our case are of the form $(\ref{appr})$, and solve the equation 
up to a vector parallel to the tangent plane of this manifold. 
To do this one can use the spectral properties of the linearization of $(\ref{prob})$, 
see Lemma $\ref{lem:varcritica}$. 
Then, see Theorem $\ref{th:rid}$, one generates a new manifold $\tilde{Z}$ close to $Z$ which represents a 
natural constraint for the Euler functional of $(\ref{problem})$, which is   
\be
 \tilde{I}_{\epsilon}\left(u\right)  = \frac{1}{2} \int_{\Omega}\left(\epsilon^ {2}\vert\nabla u\vert^ {2} + u^ {2}\right)dx - \frac{1}{p+1}\int_{\Omega}\vert u\vert^ {p+1}dx, \quad  u\in H^{1}_{D}\left(\Omega\right),
\nonumber\ee
where $H^{1}_{D}\left(\Omega\right)$ is the space of functions $H^{1}\left(\Omega\right)$ 
which have zero trace on $\partial_{D}\Omega$. 
By \emph{natural constraint} we mean a set for which constrained critical points 
of $\tilde{I}_{\epsilon}$ are true critical points. 

Now, we want to have a good control 
of the functional $\tilde{I}_{\epsilon}\mid_{\tilde{Z}}$. 
Improving the accuracy of the functions in the original manifold $Z$, 
we make $\tilde{Z}$ closer to $Z$; 
in this way the main term in the constrained functional 
will be given by $\tilde{I}_{\epsilon}\mid_{Z}$, 
see Propositions $\ref{pr:est}$, $\ref{pr:est3}$, $\ref{pr:est4}$. 
To find sufficiently good approximate solutions 
we start with those constructed in literature for the Neumann problem $(\ref{Neumann})$ 
(see Subsection $\ref{apprNeumann}$) which reveal the role of the mean curvature. 
The problem is that these functions are non zero on $\partial_{D}\Omega$, 
and even if one use cut-off functions to annihilate them 
the corresponding error turns out to be too large. 
Following the line of \cite{GMMP1} and \cite{NW}, we will use the 
\emph{projection operator} in $H^{1}\left(\Omega\right)$, 
which associates to every function in this space its closest element in 
$H^{1}_{D}\left(\Omega\right)$. 
To study the asymptotic behavior of this projection 
we will use the limit behavior of the solution $U$ to $(\ref{prob})$: 
\be       \lim_{r\rightarrow +\infty}  e^{r}  r^{\frac{n-2}{2}} U\left(r\right) = c_{n, p},       \label{limU}\ee
where $r=\vert x\vert$ and $c_{n, p}$ is a positive constant depending only on the dimension $n$ and $p$, 
together with 
\be      \lim_{r\rightarrow +\infty}\frac{U'\left(r\right)}{U\left(r\right)} =  -   \lim_{r\rightarrow +\infty}\frac{U''\left(r\right)}{U\left(r\right)} = -1,
\label{limU2}    \ee 
as it was done in some previous works, see for instance \cite{LN} and \cite{We2}. 
Moreover, we will work at a scale $d\simeq\epsilon\vert\log\epsilon\vert$, 
which is the order of the distance of the peak from $\Gamma$, see 
Remark $\ref{rem:th}$ (b). At this scale both $\partial_{N}\Omega$ 
and $\partial_{D}\Omega$ look flat; so we can identify them with the 
hypersurfaces of equations $x_{n}=0$ and 
$x_{1}\tan\alpha + x_{n}=0$, and 
their intersection 
with the set $\left\lbrace x_{1}=x_{n}=0\right\rbrace$. 
Note that $\alpha =\alpha\left(Q\right)$ is the angle between $x_{1}$ and $x_{n}$ at a fixed point $Q\in\Gamma$. 
Then we can replace $\Omega$ 
with a suitable domain $\Sigma_{D}$, 
which in particular for $0<\alpha\leq\pi$ is even with respect to the coordinate $x_{n}$, 
see the beginning of Subsections $\ref{sec:3.1}$ and $\ref{sec:3.2}$. 
Now, studying the projections in this domain, 
we will find functions which have zero $x_{n}$-derivative 
on $\left\lbrace x_{n}=0\right\rbrace\setminus\partial\Sigma_{D}$, 
which mimics the Neumann boundary condition on $\partial_{N}\Omega$. 
After analyzing carefully the projection in Subsections $\ref{sec:3.1}$, $\ref{sec:3.2}$, 
we will be able to define a family of suitable approximate solutions 
to $(\ref{problem})$ which have sufficient accuracy for our analysis, 
estimated in Propositions $\ref{pr:est}$, $\ref{pr:est3}$, $\ref{pr:est4}$. 

We can finally apply the above mentioned perturbation method to reduce the problem 
to a finite-dimensional one, and study the functional constrained on $\tilde{Z}$. 
We obtain an expansion of the energy of 
the approximate solutions, which turns out to be 
\be  \tilde{I}_{\epsilon}\left(u_{\epsilon, Q}\right) = \tilde{C}_{0} \epsilon^{n}  
         - \tilde{C}_{1}\epsilon^{n+1} H\left(Q\right) + 
             \epsilon^{n} e^{-2\frac{d_{\epsilon}}{\epsilon}\left(1+o\left(1\right)\right)}  +   
             \epsilon^{n} e^{-\frac{d_{\epsilon}}{\epsilon}\left( 1+ \frac{\sqrt{2} \tan\alpha\left(Q\right)}{\sqrt{\tan^{2}\alpha\left(Q\right) +1}} \right) \left(1+o\left(1\right)\right)}  + 
               o\left(\epsilon^{n+2}\right),  \nonumber\ee
in the case $0<\alpha <\frac{\pi}{2}$, and 
\be  \tilde{I}_{\epsilon}\left(u_{\epsilon, Q}\right) = 
   \tilde{C}_{0}\epsilon^{n}   - \tilde{C}_{1}\epsilon^{n+1}H\left(Q\right) + 
\epsilon^{n} e^{-2\frac{d_{\epsilon}}{\epsilon}\left(1 + o\left(1\right) \right)}  + 
o\left(\epsilon^{n+2}\right), \nonumber\ee 
in the case $\frac{\pi}{2}\leq\alpha <2\pi$. 
As for $(\ref{energy})$, we have that the first two terms come from the Neumann 
condition, while the others are related to the repulsive effect due to 
the Dirichlet condition. 
Let us notice that, in the first case, in the terms related to 
the Dirichlet condition appears the opening angle $\alpha$, 
whereas in the second case it does not; 
this phenomenon comes from the fact that the distance of the point $Q$ 
from the Dirichlet part $\partial_{D}\Omega$ depends on $\alpha$ only if $0<\alpha <\frac{\pi}{2}$.  

\medskip

Concerning the regularity of the solution, following the ideas in \cite{Gri}, 
it is possible to say that it is influenced by the presence of the angle. 
In fact, the solution is at least $C^{2}$ in the interior of the domain, far from the angle; 
whereas, near the angle, one can split the solution into a regular part and 
a singular one, whose regularity depends on the value of $\alpha$. 
For more details about the regularity of solutions in non-smooth domains 
we refer the reader to the book \cite{Gri}. 

The fact that the solution $u$ is $C^{2}$ in the interior of the domain 
allows to say also that it is strictly positive, by using the strong Maximum Principle. 
In fact, we have that $u\geq 0$ in the domain.  
Moreover, if there exists a point $x_{0}$ in the interior of the domain 
such that $u\left(x_{0}\right)=0$, we can consider a ball centered at $x_{0}$ 
of small radius suct that it is contained in the domain; 
since in the ball $u$ is $C^{2}$ we can conclude 
that $u$ cannot be zero in $x_{0}$.

\bigskip

The plan of the paper is the following. 
In Section $\ref{preliminaries}$ we collect some preliminary material: 
we recall the abstract variational perturbative scheme and some known results 
concerning the Neumann problem $(\ref{Neumann})$. 
In Section $\ref{apprsol}$ we construct a model domain to deal with the interface, 
analyze the asymptotics of projections in $H^{1}$ and then construct 
approximate solution to $(\ref{problem})$. 
Finally in Section $\ref{proof}$ we expand the functional on the natural constraint, 
prove the existence of critical points and deduce Theorem $\ref{th:solution}$. 

\medskip 

\subsection*{Notation} 
Generic fixed constant will be denoted by $C$, 
and will be allowed to vary within a single line or formula. 
The symbol
$o\left(t\right)$ 
will denote quantities for which 
$\frac{o\left(t\right)}{\vert t\vert}$ tends to zero 
as the argument $t$ goes to zero or to infinity. 
We will often use the notation $d\left(1+o\left(1\right)\right)$, 
where $o\left(1\right)$ stands for a quantity which tends to zero as 
$d\rightarrow +\infty$.

\section{Preliminaries}\label{preliminaries} 
We want to find solutions to $(\ref{problem})$ with a specific 
asymptotic profile, so it is convenient to 
make the change of variables $x\mapsto\epsilon x$, 
and study $(\ref{problem})$ in the dilated domain 
\be       \Omega_{\epsilon}:=\frac{1}{\epsilon}\Omega .             \nonumber\ee 
Then the problem becomes 
\bern
\left\{ 
\begin{array}{lll} 
         -\Delta u + u = u^ {p}   \quad & \mathrm{in\ } \Omega_{\epsilon},      \\
          \frac{\partial u}{\partial \nu} = 0  \quad \mathrm{on\ } \partial_{N}\Omega_{\epsilon},   &
            u=0 \quad  \mathrm{on\ } \partial_{D}\Omega_{\epsilon},  \\ 
          u>0   \quad    &  \mathrm{in\ } \Omega_{\epsilon},
 \end{array} 
\right.         
 \label{problem1}\eern 
where $\partial_{N}\Omega_{\epsilon}$ and $\partial_{D}\Omega_{\epsilon}$ 
stand for the dilations of $\partial_{N}\Omega$ and $\partial_{D}\Omega$ respectively. 
Moreover we denote by $\Gamma_{\epsilon}$ 
the intersection of the closures of $\partial_{N}\Omega_{\epsilon}$ and $\partial_{D}\Omega_{\epsilon}$. 

Solutions of  $(\ref{problem1})$ can be found as critical points  
of the Euler-Lagrange functional
\be
 I_{\epsilon}\left(u\right)  = \frac{1}{2} \int_{\Omega_{\epsilon}}\left(\vert\nabla u\vert^ {2} + u^ {2}\right)dx - \frac{1}{p+1}\int_{\Omega_{\epsilon}}\vert u\vert^ {p+1}dx, \quad  u\in H^{1}_{D}\left(\Omega_{\epsilon}\right).
  \nonumber \ee
Here $H^{1}_{D}\left(\Omega_{\epsilon}\right)$ denotes the space of functions in $H^{1}\left(\Omega_{\epsilon}\right)$ 
with zero trace on $\partial_{D}\Omega_{\epsilon}$. 

In the next subsection we introduce the abstract perturbation method 
which takes advantage of the variational structure of the problem, 
and allows us to reduce it to a finite dimensional one. 
We refer the reader mainly to \cite{AM}, \cite{Ma} 
and the bibliography therein for the abstract method. 
In our case we will use some small modifications of the arguments in the latter references 
which can be found in Subsection $2.1$ of \cite{GMMP1}.

\subsection{Perturbation in critical point theory} \label{perturbation}
In this subsection we recall some results 
about the existence of critical points 
for a class of functionals which are perturbative in nature. 
Given an Hilbert space $\emph{H}$, 
which might depend on the perturbation parameter $\epsilon$, 
we consider manifolds embedded smoothly in $\emph{H}$, for which
\begin{itemize}
\item[i)]  there exists a smooth finite-dimensional manifold $Z_{\epsilon}\subseteq\emph{H}$ 
and $C, r>0$ such that for any $z\in Z_{\epsilon}$, 
the set $Z_{\epsilon}\cap B_{r}\left(z\right)$ can be parametrized by a map on $B^{\R^{d}}_{1}$ 
whose $C^{3}$ norm is bounded by $C$.
\end{itemize} 
Moreover we are interested in functionals $I_{\epsilon}:\emph{H}\rightarrow\R$ of class $C^{2, \gamma}$ 
which satisfy the following properties:
\begin{itemize}
\item[ii)] there exists a continuous function $f:\left(0, \epsilon_{0}\right)\rightarrow\R$ 
with $\lim_{\epsilon\rightarrow 0}f\left(\epsilon\right)=0$ such that 
$\norm{I'_{\epsilon}(z)}\leq f\left(\epsilon\right)$ for every $z\in Z_{\epsilon}$; 
moreover $\norm{I''_{\epsilon}(z)\left[ q\right] }\leq f\left(\epsilon\right)\norm{q}$ 
for every $z\in Z_{\epsilon}$ and every $q\in T_{z}Z_{\epsilon}$;
\item[iii)] there exist $C, \gamma\in\left(0, 1\right]$, $r_{0}>0$ such that $\norm{I''_{\epsilon}}_{C^{\gamma}}\leq C$ 
in the subset $\left\lbrace u:dist\left(u, Z_{\epsilon}\right)<r_{0}\right\rbrace$; 
\item[iv)] letting $P_{z}:\emph{H}\rightarrow\left(T_{z}Z_{\epsilon}\right)^{\perp}$, for every $z\in Z_{\epsilon}$, 
be the projection onto the orthogonal complement of $T_{z}Z_{\epsilon}$, 
there exists $C>0$, independent of $z$ and $\epsilon$, such that $P_{z}I''_{\epsilon}(z)$, 
restricted to $\left(T_{z}Z_{\epsilon}\right)^{\perp}$, is invertible from $\left(T_{z}Z_{\epsilon}\right)^{\perp}$ 
into itself, and the inverse operator satisfies $\norm{\left( P_{z}I''_{\epsilon}(z)\right)^{-1}}\leq C$.
\end{itemize}   
We set $W=\left(T_{z}Z_{\epsilon}\right)^{\perp}$, and look for critical points of  $I_{\epsilon}$ in the form $u=z+w$ 
with $z\in Z_{\epsilon}$ and $w\in W$. 
If $P_{z}:\emph{H}\rightarrow W$ is as in $iv)$, 
the equation $I'_{\epsilon}\left(z+w\right) = 0$ 
is equivalent to the following system
\bern
\left\{ 
\begin{array}{ll} 
          P_{z}I'_{\epsilon}\left(z+w\right) = 0  \qquad & \mathrm{\left(\textit{the\ auxiliary\ equation}\right),}      \\
           \left(Id-P_{z}\right) I'_{\epsilon}\left(z+w\right) = 0  \quad & \mathrm{\left(\textit{the\ bifurcation\ equation}\right).}
 \end{array} 
\right.         
 \label{aux_bif}\eern 

\begin{proposition} 
\textsl{(See Proposition $2.1$ in \cite{GMMP1})} 
Let $i)-iv)$ hold true. 
Then there exists $\epsilon_{0}>0$ with the following property: 
for all $\vert\epsilon\vert<\epsilon_{0}$ and for all $z\in Z_{\epsilon}$, 
the auxiliary equation in $(\ref{aux_bif})$ has a unique solution $w=w_{\epsilon}(z)\in W$, 
which is of class $C^{1}$ with respect to $z\in Z_{\epsilon}$ 
and  such that $\norm{w_{\epsilon}(z)}\leq C_{1}f\left(\epsilon\right)$ as $\vert\epsilon\vert\rightarrow 0$, 
uniformly with respect to $z\in Z_{\epsilon}$. 
Moreover the derivative of $w$ with respect to $z$,  $w'_{\epsilon}$ 
satisfies the bound $\norm{w'_{\epsilon}(z)}\leq CC_{1}f\left(\epsilon\right)^{\gamma}$.
\label{pr:fi}\end{proposition} 

\noindent We shall now solve the bifurcation equation in $(\ref{aux_bif})$. 
In order to do this, let us define the \textit{reduced functional} $\mathbf{I}_{\epsilon} : Z_{\epsilon}\rightarrow\R$ by setting 
$\mathbf{I}_{\epsilon}(z) = I_{\epsilon}(z+ w_{\epsilon}(z))$. 

\begin{theorem}
\textsl{(See Proposition $2.3$ in \cite{GMMP1})} 
Suppose we are in the situation of Proposition $\ref{pr:fi}$, 
and let us assume that $\mathbf{I}_{\epsilon}$ has, for $\vert\epsilon\vert$ sufficiently small, 
a stationary point $z_{\epsilon}$. 
Then $u_{\epsilon} = z_{\epsilon} + w(z_{\epsilon})$ is a critical point of $I_{\epsilon}$. 
Furthermore, there exist $\tilde{c}, \tilde{r}>0$ such that if $u$ is a critical point of $I_{\epsilon}$ 
with $dist\left(u, Z_{\epsilon, \tilde{c}}\right)<\tilde{r}$, where 
$Z_{\epsilon, \tilde{c}}=\left\lbrace z\in Z_{\epsilon}:dist\left(z, \partial Z_{\epsilon}\right)>\tilde{c}\right\rbrace$, 
then $u$ has to be of the form $z_{\epsilon} + w(z_{\epsilon})$ for some $z_{\epsilon}\in Z_{\epsilon}$.
\label{th:rid}
\end{theorem}

\subsection{Approximate solutions for $(\ref{problem})$ with Neumann conditions} \label{apprNeumann}
In this subsection we introduce some convenient coordinates which stretch the boundary 
and we recall some results from \cite{AM} and \cite{GMMP1} 
concerning approximate solutions to the Neumann problem.

\medskip

First of all it can be shown that 
near a generic point $Q\in\Gamma$ 
the boundary of $\Omega$ can be described by a coordinate system $y=\left(y_{1}, \ldots , y_{n}\right)$ such that 
\begin{itemize}
\item[(a)] $\partial_{N}\Omega$ coincides with $\left\lbrace y_{n}=0\right\rbrace$,
\item[(b)] $\partial_{D}\Omega$ coincides with $\left\lbrace y_{1}\tan\alpha+y_{n}=0\right\rbrace$, 
where $\alpha=\alpha\left(Q\right)$ is the opening angle of $\Gamma$ at $Q$,
\item[(c)] the corresponding metric coefficients are given by $g_{ij}=\delta_{ij}+o\left(\epsilon\right)$.
\end{itemize}
For further details we refer the reader to \cite{Di}. 

\begin{remark}
\begin{itemize}
\item[(i)] We stress that, in the new coordinates $y$, the origin parametrizes 
the point $Q$, and those functions decaying as 
$\vert y\vert\rightarrow +\infty$ will \emph{concentrate} near $Q$. 
\item[(ii)] It is also useful to understand how the metric coefficients $g_{ij}$ 
vary with $Q$. Notice that condition $(c)$ says that the deviation from the 
Kronecker symbols is of order $\epsilon$, and we are working in a domain scaled 
of $\frac{1}{\epsilon}$; hence a variation of order $1$ of $Q$ corresponds to a 
variation of order $\epsilon$ in the original domain. 
Therefore, a variation of order $1$ in $Q$ yields a difference of order 
$\epsilon^{2}$ in $g_{ij}$, and precisely
\be    \frac{\partial g_{ij}}{\partial Q} = 
                 o\left( \epsilon^{2}\vert y\vert^{2} \right), 
\nonumber\ee 
with a similar estimate for the derivatives of the inverse coefficients $g^{ij}$. 
For more details see the end of Subsection $9.2$ in \cite{AM}. 
\end{itemize}
\label{rem:gij}\end{remark}

\noindent Suppose that this coordinate system $y$ is defined in $B_{\mu_{0}}\left(Q\right)$, 
with $\mu_{0}>0$ sufficiently small. 
Now, in this set of coordinates we choose a cut-off function $\chi_{\mu_{0}}$ with the following properties
\bern
\left\{ 
\begin{array}{lll} 
           \chi_{\mu_{0}}\left(x\right) = 1 \qquad & \mathrm{in\ }   B_{\frac{\mu_{0}}{4}}\left(Q\right),    \\
            \chi_{\mu_{0}}\left(x\right) = 0 \qquad  & \mathrm{in\ } \R^{n}\setminus B_{\frac{\mu_{0}}{2}}\left(Q\right), \\
            \vert\nabla\chi_{\mu_{0}}\vert + \vert\nabla^{2}\chi_{\mu_{0}}\vert\leq C \qquad  & \mathrm{in\ } B_{\frac{\mu_{0}}{2}}\left(Q\right)\setminus B_{\frac{\mu_{0}}{4}}\left(Q\right),
 \end{array} 
\right.         
 \nonumber\eern  
and we define the approximate solution $\bar{u}_{\epsilon, Q}$ as 
\be    \bar{u}_{\epsilon, Q}\left(y\right)  :=  \chi_{\mu_{0}}\left(\epsilon y\right)\left(U_{Q}\left(y\right)+\epsilon w_{Q}\left(y\right)\right), 
\label{function}\ee
where $U_{Q}\left(y\right)=U\left(y-Q\right)$ and 
$w_{Q}$ is a suitable function obtained in Subsection $2.2$ of \cite{GMMP1} by a small 
modifications of Lemma $9.3$ in \cite{AM}, 
satisfying the following estimate
\be      \vert  w_{Q}\left(y\right)\vert +  \vert \nabla w_{Q}\left(y\right)\vert  + \vert \nabla^{2} w_{Q}\left(y\right)\vert  
            \leq  C_{\Omega} \left( 1 + \vert y\vert^{K}\right) e^{-\vert y\vert},        \label{est w}\ee
where $C_{\Omega}$ and $K$ are constants depending on $\Omega$, $H$, $n$ and $p$. 

The next result collects estimates obtained following the same arguments of 
Lemmas $9.4$, $9.7$ and $9.8$ in \cite{AM}. 
\begin{proposition} 
There exist $C, K>0$ such that for $\epsilon$ small the following estimates hold
\be
\vert \frac{\partial\bar{u}_{\epsilon, Q}}{\partial\nu_{g}}\vert \left(y\right) \leq  
\left\{ 
\begin{array}{ll} 
           C\epsilon^{2} \left( 1 + \vert y\vert^{K}\right) e^{-\vert y\vert}    \qquad & \mathrm{for\ }  
            \vert y\vert\leq\frac{\mu_{0}}{4\epsilon},    \\
            C e^{-\frac{1}{C\epsilon}} \qquad  & \mathrm{for\ }  \frac{\mu_{0}}{4\epsilon}\leq\vert y\vert\leq\frac{\mu_{0}}{2\epsilon};
\end{array} 
\right.         
\nonumber\ee
\be
\vert -\Delta_{g}\bar{u}_{\epsilon, Q} +\bar{u}_{\epsilon, Q} -  \bar{u}_{\epsilon, Q}^{p}  \vert \left(y\right) \leq  
\left\{ 
\begin{array}{ll} 
           C\epsilon^{2} \left( 1 + \vert y\vert^{K}\right) e^{-\vert y\vert}    \qquad & \mathrm{for\ }  
            \vert y\vert\leq\frac{\mu_{0}}{4\epsilon},    \\
            C e^{-\frac{1}{C\epsilon}} \qquad  & \mathrm{for\ }  \frac{\mu_{0}}{4\epsilon}\leq\vert y\vert\leq\frac{\mu_{0}}{2\epsilon};
\end{array} 
\right.         
\nonumber\ee
\be      I_{\epsilon, N}\left(\bar{u}_{\epsilon, Q}\right) = \tilde{C}_{0} -  \tilde{C}_{1}\epsilon H\left(\epsilon Q\right) + o\left(\epsilon^{2}\right);      \qquad 
       \frac{\partial}{\partial Q}  I_{\epsilon, N}\left(\bar{u}_{\epsilon, Q}\right) =  -\tilde{C}_{1}\epsilon^{2} H'\left(\epsilon Q\right) + o\left(\epsilon^{2}\right),  \nonumber\ee
where 
\be    \tilde{C}_{0}=\left(\frac{1}{2}-\frac{1}{p+1}\right) \int_{\R^{n}_{+}} U^{p+1} dy,      \qquad 
          \tilde{C}_{1}=\left(\int_{0}^{\infty} r^{n} U_{r}^{2} dr \right) \int_{S^{n}_{+}} y_{n}\vert y' \vert^{2} d\sigma .   \nonumber\ee   
\label{prop:estNeumann} \end{proposition}

\noindent An immediate consequence of this proposition is that 
\be     \norm{I'_{\epsilon}\left(\bar{u}_{\epsilon, Q}\right)}   \leq    C\epsilon^{2}     \qquad     
            \mathrm{for\ all\ }      Q\in\partial_{N}\Omega_{\epsilon}    \mathrm{\ such\ that\ }     
             dist\left(Q, \Gamma_{\epsilon}\right) \geq    \frac{\mu_{0}}{\epsilon},    
\label{estNeumann}\ee 
where $C>0$ is some fixed constant and $\mu_{0}$ is as before.

\section{Approximate solutions to $(\ref{problem1})$} \label{apprsol}
To construct good approximate solutions to $(\ref{problem1})$, 
we will start from a family of known functions which constitute 
good approximate solutions to $(\ref{problem1})$ when we impose 
pure Neumann boundary conditions. 
Since we have to take into account the effect of the Dirichlet boundary conditions, 
we will modify these functions in a convenient way. 
Following the line of \cite{GMMP1} and \cite{NW}, 
we will use the \textit{projection operator} onto $H^{1}_{D}\left(\Omega_{\epsilon}\right)$, 
which associates to every element in $H^{1}\left(\Omega_{\epsilon}\right)$ 
its closest point in $H^{1}_{D}\left(\Omega_{\epsilon}\right)$. 
Explicitly, this is constructed subtracting to any given $u\in H^{1}\left(\Omega_{\epsilon}\right)$ 
the solution to 
\bern
\left\{ 
\begin{array}{lll} 
         -\Delta v + v = 0   \quad & \mathrm{in\ } \Omega_{\epsilon},      \\
           v=u \quad     &    \mathrm{on\ } \partial_{D}\Omega_{\epsilon},  \\
          \frac{\partial v}{\partial \nu} = 0  \quad      &   \mathrm{on\ } \partial_{N}\Omega_{\epsilon}.
 \end{array} 
\right.         
 \label{dirichlet}\eern 
This solution can be found variationally by looking at the following minimum problem
\be     \inf_{v=u \mathrm{\ on\ } \partial_{D}\Omega_{\epsilon}} \left\lbrace \int_{\Omega_{\epsilon}}\left(\vert\nabla v\vert^ {2} + v^ {2}\right)dx \right\rbrace . \nonumber\ee
Instead of studying $(\ref{dirichlet})$ directly, 
it is convenient to modify the domain in order that the region of the boundary 
near $\Gamma_{\epsilon}$ becomes flat. 
We fix $Q\in\Gamma_{\epsilon}$ and consider the opening angle of $\Gamma_{\epsilon}$ at $Q$, 
$\alpha =\alpha\left(Q\right)$. 
Since the construction of this new domain is different for $0<\alpha\leq\pi$ and $\pi <\alpha <2\pi$, 
we will study separately the two cases in the following two subsections.

\subsection{Case  $0<\alpha\leq\pi$} \label{sec:3.1}
For technical reasons we construct a domain $\Sigma$ in the following way: 
we consider two hypersurfaces defined by the equations $x_{1}\tan\alpha +x_{n}=0$ and $x_{1}\tan\alpha -x_{n}=0$, 
which obviously intersect at $\left\lbrace x_{1}=x_{n}=0\right\rbrace$. 
Then we close the domain between the two hypersurfaces 
with $x_{1}<0$ if $0<\alpha <\frac{\pi}{2}$ 
and with $x_{1}>0$ if $\frac{\pi}{2}\leq\alpha\leq\pi$ 
with a smooth surface, 
in such a way that the scaled domain 
\be   \Sigma_{D} = D \Sigma ,     \label{domain}\ee
defined for a large number $D$, contains a sufficiently large cube. 
In $\Sigma_{D}$ we denote by $\Gamma_{D}$ the singularity, 
which lies on $\left\lbrace x_{1}=x_{n}=0\right\rbrace$. 
The following figure represents a section of the domain in the plane $x_{1}$, $x_{n}$.
\begin{center}
\includegraphics[scale=0.5]{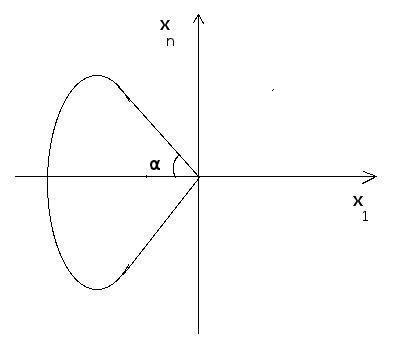} 
\end{center}
The advantage of dealing with this set is that if we solve a Dirichlet problem in $\Sigma_{D}$ 
with data even in $x_{n}$, then for suitable boundary conditions the solution in the upper part 
$\Sigma_{D}\cap\left\lbrace x_{n}>0\right\rbrace$ 
will be qualitatively similar to that of $(\ref{dirichlet})$.

\bigskip

Our next goal is to consider the following problem
\bern
\left\{ 
\begin{array}{ll} 
         -\Delta \tilde{\varphi} + \tilde{\varphi} = 0   \quad & \mathrm{in\ } \Sigma_{dD},      \\
           \tilde{\varphi} = U\left(\cdot - dQ_{0}\right) \quad     &    \mathrm{on\ } \partial\Sigma_{dD},  
\end{array} 
\right.         
 \label{dirichlet1}\eern      
where $Q_{0}=\left(-1, 0, \cdots, 0\right)$. 
By a scaling of variables, this problem is equivalent to 
\bern
\left\{ 
\begin{array}{ll} 
         -\frac{1}{d^{2}}\Delta \varphi + \varphi = 0   \quad & \mathrm{in\ } \Sigma_{D},      \\
           \varphi = U\left(d\left(\cdot -Q_{0}\right)\right)  \quad     &    \mathrm{on\ } \partial\Sigma_{D},  
\end{array} 
\right.         
 \label{varphi}\eern

\subsubsection{Asymptotic analysis of $(\ref{varphi})$} \label{sec:asympt}
First of all we need to know if $(\ref{varphi})$ is solvable. 
It follows from Lemma $3.1$ in \cite{GMMP1}; in fact, 
making a modification of some arguments in \cite{Gri}, 
they construct barrier functions for the operators 
$\Delta$ and $-\Delta +1$ at all boundary points of the set $\Sigma$. 
This guarantees, via the classical Perron method, the existence of a solution for the problem $(\ref{varphi})$.

If we consider the function $\phi =-\frac{1}{d}\log\varphi$, 
then $\phi$ satisfies 
\bern
\left\{ 
\begin{array}{ll} 
         \frac{1}{d}\Delta\phi - \vert\nabla\phi\vert ^{2} + 1 = 0   \quad & \mathrm{in\ } \Sigma_{D},      \\
        \phi =  -\frac{1}{d}\log\left(U\left(d\left(\cdot -Q_{0}\right) \right) \right)  \quad  &  \mathrm{on\ } \partial\Sigma_{D}.    
 \end{array} 
\right.         
 \label{logphi}\eern 
Using the limit behavior of the function $U$ given by $(\ref{limU})$, 
it is easy to show the following: 
\begin{lemma} 
For any fixed constant $D>0$ we have that 
\be    -\frac{1}{d}\log\left(U\left(d\left(\cdot -Q_{0}\right) \right) \right)\rightarrow \vert \cdot -Q_{0}\vert     \qquad   \mathrm{uniformly\ on\ } \partial\Sigma_{D}   \label{bordo}\ee
as $d\rightarrow +\infty$. 
\label{lem:bordo} 
\end{lemma} 
Since Lemma $\ref{lem:bordo}$ states that the boundary datum is everywhere close to the function $\vert x-Q_{0}\vert$, 
it is useful to consider the following auxiliary problem
\bern
\left\{ 
\begin{array}{ll} 
         \frac{1}{d}\Delta\phi - \vert\nabla\phi\vert ^{2} + 1 = 0   \quad & \mathrm{in\ } \Sigma_{D},      \\
        \phi = \vert x-Q_{0}\vert  \quad  &  \mathrm{on\ } \partial\Sigma_{D}.    
 \end{array} 
\right.         
 \label{phi}\eern 

\begin{lemma} 
Let $D>1$ be a fixed constant. 
Then, when $d\rightarrow\infty$, problem $(\ref{phi})$ has a unique solution $\phi ^{d}$, 
which is everywhere positive, and which more precisely satisfies the estimates 
\be     \frac{\tan\alpha}{\sqrt{\tan^{2}\alpha +1}}<\phi^{d}\left(x\right)<C    \qquad \mathrm{in\ }   \Sigma_{D}, 
\label{stima}\ee
if $0<\alpha<\frac{\pi}{2}$, and 
\be     1<\phi^{d}\left(x\right)<C    \qquad \mathrm{in\ }   \Sigma_{D}, 
\label{stima'}\ee
if $\frac{\pi}{2}\leq\alpha\leq\pi$, 
where $C$ depends only on $D$ and $\Sigma$.
\label{lem:stima}\end{lemma} 

\begin{proof} 
Applying the transformation inverse to the one at the beginning of this subsection 
and using the existence of barrier functions for the operator $-\Delta +1$, as shown in \cite{GMMP1}, Lemma $3.1$, 
we get existence. 
Uniqueness and positivity of $\phi ^{d}$ follows from the maximum principle. 

To prove the estimates $(\ref{stima})$ and $(\ref{stima'})$, we can reason as in \cite{GMMP1}, Lemma $3.4$, or 
in \cite{NW}, Lemma $4.2$. 
In the case $0<\alpha<\frac{\pi}{2}$, 
we have that $\phi^{d}_{-}\left(x\right)\equiv\frac{\tan\alpha}{\sqrt{\tan^{2}\alpha +1}}$ in $\Sigma_{D}$ 
is a subsolution to $(\ref{phi})$, since $dist\left(Q_{0}, \partial\Sigma_{D}\right)=\frac{\tan\alpha}{\sqrt{\tan^{2}\alpha +1}}$; 
whereas, in the case $\frac{\pi}{2}\leq\alpha\leq\pi$, 
we have that $dist\left(Q_{0}, \partial\Sigma_{D}\right)=1$, 
and then the subsolution is given by $\phi^{d}_{-}\left(x\right)\equiv 1$. 
Moreover, in both the two cases, the function $\phi^{d}_{+}\left(x\right)=C+x_{1}$ is a supersolution for $C$ sufficiently large. 
Then our claim follows.  
\end{proof}

We next show some pointwise bounds on $\phi^{d}$, which in particular imply 
a control on the gradient within some region in the boundary of $\Sigma_{D}$. 
We obtain gradient bounds only near smooth parts of the boundary, away from the singularity $\Gamma_{D}$. 

\begin{lemma} 
Let $D>1$ be as in Lemma $\ref{lem:stima}$. 
Then, there exists a constant $C>0$ such that for any $\sigma>0$ sufficiently small 
there exist $\bar{\delta}>0$ and $d_{\sigma}>0$ so large that
\be     \vert \phi^{d}\left(x\right) -  \phi^{d}\left(z_{x}\right)\vert  \leq   C  \vert x- z_{x} \vert ,   \qquad   z_{x}\in\partial\Sigma_{D}, 
           dist\left(z_{x}, D\Gamma_{D}\right)\geq \sigma ,  \vert x- z_{x} \vert\leq\bar{\delta},  d\geq d_{\sigma}.   
\nonumber\ee
In the above formula $z_{x}$ denotes the point in $\partial\Sigma_{D}$ closest to $x$. 
\label{lem:stima1}\end{lemma} 

\begin{proof} 
Let us first consider the case $0<\alpha<\frac{\pi}{2}$. 
Let us fix $\sigma>0$ small and consider, for every $0<\delta<\bar{\delta}=\sigma\tan\alpha$, 
the points $x\in\Sigma_{D}$ of the form $z+\delta\nu\left(z\right)$, 
where $z\in\partial\Sigma_{D}$ and $\nu\left(z\right)$ is the inner unit normal at $z$. 
Note that there is no problem in the representation of $x$ if $dist\left(z, D\Gamma_{D}\right)\geq\sigma$; 
whereas if $dist\left(z, D\Gamma_{D}\right)<\sigma$, we follow the inner normal direction 
given by $\nu\left(z\right)$ and stop at $x_{n}=0$ if we reach this hyperplane at a distance 
from the boundary smaller than $\bar{\delta}$. 
Let us call $\Lambda_{\delta}$ this set of points $x\in\Sigma_{D}$ 
at distance $\delta$ from the boundary. 
Note that the $\Lambda_{\delta}$'s are all disjoint as $\delta$ varies in $\left[0, \bar{\delta}\right]$. 
Now in $\Lambda_{\delta}$ we can define the functions
\bern    \phi_{1}\left(x\right)=\vert z_{1}\left(x\right)-Q_{0}\vert  + M\delta_{1}\left(x\right), \nonumber\\
             \phi_{2}\left(x\right)=\vert z_{2}\left(x\right)-Q_{0}\vert  + M\delta_{2}\left(x\right), \nonumber
\eern
where $z_{1}\left(x\right)$, $z_{2}\left(x\right)$ are the points in $\partial\Sigma_{D}$ closest to $x$ 
with the $n$-th coordinate respectively positive and negative; $\delta_{1}\left(x\right)$, $\delta_{2}\left(x\right)$ 
give the distance of $x$ from $z_{1}\left(x\right)$, $z_{2}\left(x\right)$. 
If we set 
\be     \hat{\phi}^{d}_{+}\left(x\right)=\min\left\lbrace\phi_{1}\left(x\right), \phi_{2}\left(x\right)\right\rbrace ,  \nonumber\ee
we choose the constant $M$ so large that $\hat{\phi}^{d}_{+}\left(x\right)>\phi^{d}\left(x\right)$ 
when $x\in\left\lbrace z+\bar{\delta}\nu\left(z\right):z\in\partial\Sigma_{D} \right\rbrace$. 
The existence of such constant $M$ is guaranteed by Lemma $\ref{lem:stima}$. 

Next we consider a smooth function $\rho\in\mathcal{C}^{\infty}_{0}\left(\R^{n}\right)$, 
such that $supp\rho\subset B_{1}\left(0\right)$, and $\int_{\R^{n}}\rho\left(x\right)dx=1$. 
Moreover we define the function 
\be   \lambda\left(x\right)=-\frac{2}{\bar{\delta}}\delta^{2}\left(x\right)+2\delta\left(x\right), \qquad   
\mathrm{for\ } x\in\left\lbrace z+\delta\nu\left(z\right): z\in\partial\Sigma_{D}, \delta\in\left[0, \bar{\delta}\right] \right\rbrace .     
\nonumber\ee
Then we construct a mollifiers 
\be    \rho_{\lambda\left(x\right)}\left(y\right)=\frac{1}{\lambda^{n}\left(x\right)}\rho\left(\frac{y}{\lambda\left(x\right)}\right),   \label{mollifier} \ee 
in such a way that the support of each $\rho_{\lambda\left(x\right)}$ depends on the point $x$, 
and, in particular, it shrinks to a point when we are close to the boundary. 

Finally we regularize $\hat{\phi}^{d}_{+}$ using the convolution with the mollifiers defined in $(\ref{mollifier})$. 
Then we obtain the following smooth function
\be   \phi^{d}_{+}\left(x\right) = \left(\hat{\phi}^{d}_{+}\ast\rho_{\lambda\left(\cdot\right)}\right)\left(x\right) 
                     =  \int_{\R^{n}}  \hat{\phi}^{d}_{+}\left(x-y\right)  \rho_{\lambda\left(x\right)}\left(y\right)  dy.   \nonumber\ee 
It is easy to see that, for $i=1,\ldots,n-1$,  
\bern          \frac{\partial\phi_{1}}{\partial x_{i}}  &<&   o\left(1\right) -\frac{1}{C}M,   \label{fi1}\\
                    \frac{\partial\phi_{2}}{\partial x_{i}} &<&   o\left(1\right) -\frac{1}{C}M.    \label{fi2} 
\eern
Moreover, using $(\ref{fi1})$ and $(\ref{fi2})$, we have that 
\bern 
     \frac{\partial\phi^{d}_{+}}{\partial x_{i}}  &=&  \int_{\R^{n}}\frac{\partial\hat{\phi}^{d}_{+}}{\partial x_{i}}\left(x-y\right)\rho_{\lambda\left(x\right)}\left(y\right)dy 
      +   \int_{\R^{n}}\hat{\phi}^{d}_{+}\left(x-y\right)\frac{\partial\rho_{\lambda\left(x\right)}}{\partial\lambda}\left(y\right) 
      \frac{\partial\lambda}{\partial x_{i}}\left(x\right)  dy   \nonumber\\
       &\leq &    o\left(1\right) -\frac{1}{C}M  +    
                 \int_{\R^{n}}\hat{\phi}^{d}_{+}\left(x-y\right)\frac{\partial\rho_{\lambda\left(x\right)}}{\partial\lambda}\left(y\right) 
      \frac{\partial\lambda}{\partial x_{i}}\left(x\right)  dy. 
\label{partial}\eern 
Now we need an estimate for the last term in $(\ref{partial})$, let us call it $A$. 
If we add and subtract $\hat{\phi}^{d}_{+}\left(x\right)$ in the integral, we obtain
\bern   
    A    &=&     \int_{\R^{n}}\hat{\phi}^{d}_{+}\left(x\right)\frac{\partial\rho_{\lambda\left(x\right)}}{\partial\lambda}\left(y\right) 
      \frac{\partial\lambda}{\partial x_{i}}\left(x\right)  dy   +   
             \int_{\R^{n}}\left[\hat{\phi}^{d}_{+}\left(x-y\right)-\hat{\phi}^{d}_{+}\left(x\right)\right] \frac{\partial\rho_{\lambda\left(x\right)}}{\partial\lambda}\left(y\right)\frac{\partial\lambda}{\partial x_{i}}\left(x\right)  dy  \nonumber\\
          &=&  \frac{\partial\lambda}{\partial x_{i}}\left(x\right) \int_{\R^{n}}\left[\hat{\phi}^{d}_{+}\left(x-y\right)-\hat{\phi}^{d}_{+}\left(x\right)\right] \frac{\partial\rho_{\lambda\left(x\right)}}{\partial\lambda}\left(y\right)  dy;      \nonumber\eern
in the last step we have used the fact that $\hat{\phi}^{d}_{+}\left(x\right)$ and 
$\frac{\partial\lambda}{\partial x_{i}}\left(x\right)$ do not depend on $y$, 
and the fact that $\int_{\R^{n}}\frac{\partial\rho_{\lambda\left(x\right)}}{\partial\lambda}\left(y\right)dy=\frac{\partial}{\partial\lambda}\int_{\R^{n}}\rho_{\lambda\left(x\right)}\left(y\right)dy=0$, 
since $\int_{\R^{n}}\rho_{\lambda\left(x\right)}\left(y\right)dy=1$, for every $\lambda>0$. 
Now, from $(\ref{mollifier})$, a simple computation yields 
\be       \frac{\partial\rho_{\lambda\left(x\right)}}{\partial\lambda}\left(y\right) =   
               -n\lambda^{-n-1}\left(x\right)\rho\left(\frac{y}{\lambda\left(x\right)}\right)   -
                    \lambda^{-n-2}\left(x\right)y \nabla\rho\left(\frac{y}{\lambda\left(x\right)}\right).    \nonumber\ee
Then, using the fact that $\frac{\partial\lambda}{\partial x_{i}}\left(x\right)\simeq -Cx_{i}$, for some positive constant $C$, 
and making the change of variable $y=\lambda\left(x\right)z$, we have
\be    A   = C\lambda^{-1}\left(x\right)x_{i}  
                \int_{\R^{n}}\left[\hat{\phi}^{d}_{+}\left(x-\lambda\left(x\right)z\right)-\hat{\phi}^{d}_{+}\left(x\right)\right]    \cdot 
                  \left[\rho\left(z\right) + z\nabla\rho\left(z\right)\right] dz.       \label{partial2}\ee
Since $\hat{\phi}^{d}_{+}$ is a Lipschitz function, from $(\ref{partial2})$ we get that 
\be      A \leq   C x_{i} \int_{\R^{n}}   \vert z\vert \cdot  \left[\rho\left(z\right) + z\nabla\rho\left(z\right)\right] dz,      \nonumber\ee
and then $A\leq o\left(1\right)$. 
It follows that, for $M$ sufficiently large, the norm of $\nabla\phi^{d}_{+}$ 
can be arbitrarily big on its domain. 
By $(\ref{stima})$, if $M$ is large then $\phi^{d}_{+}$ is everywhere bigger than 
$\phi^{d}$ on $\Sigma_{D}\cap\left\lbrace dist\left(\cdot , \partial\Sigma_{D}\right)=\bar{\delta}\right\rbrace$, 
so $\phi^{d}_{+}$ is a supersolution of $(\ref{phi})$ in 
$\Sigma_{D}\cap\left\lbrace dist\left(\cdot , \partial\Sigma_{D}\right)<\bar{\delta}\right\rbrace$. 

On the other hand, we claim that the function $\phi^{d}_{-}\left(x\right)=\vert x-Q_{0}\vert$ 
is a subsolution of $(\ref{phi})$ in $\Sigma_{D}\cap\left\lbrace dist\left(\cdot , \partial\Sigma_{D}\right)<\bar{\delta}\right\rbrace$. 
In fact, if we consider the set $\Sigma_{D}\setminus B_{\tilde{\delta}\left(d\right)}\left(Q_{0}\right)$, 
where $\tilde{\delta}\left(d\right)$ is a small positive number depending on $d$, 
we can see by easy computation that here $\phi^{d}_{-}$ satisfies 
\be     \frac{1}{d}\Delta\phi^{d}_{-} - \vert\nabla\phi^{d}_{-}\vert ^{2} + 1 =   \frac{n-1}{d\vert x-Q_{0}\vert}.    \nonumber\ee
Moreover, since $\phi^{d}$ is positive, we can choose $\tilde{\delta}\left(d\right)$ sufficiently small 
so that $\phi^{d}_{-}<\phi^{d}$. 
Hence we obtain that $\phi^{d}_{-}\leq\phi^{d}$ in the closure of 
$\Sigma_{D}\cap\left\lbrace dist\left(\cdot , \partial\Sigma_{D}\right)<\bar{\delta}\right\rbrace$. 

Finally, the conclusion follows from the fact that $\phi^{d}_{-}$ and $\phi^{d}_{+}$ 
coincide on the set 
\be   \left\lbrace x\in\partial\Sigma_{D}:dist\left(x, D\Gamma_{D}\right)\geq\sigma\right\rbrace \nonumber\ee 
and that we have uniform bounds on the gradient here, independently on $d$. 

In the case $\frac{\pi}{2}\leq\alpha\leq\pi$, we can repeat essentially the same construction 
of the proof of Lemma $3.5$ in \cite{GMMP1} and obtain the same conclusion.   
\end{proof}

Using the same arguments as in Lemma $3.6$ in \cite{GMMP1} 
we are able to extend the gradient estimate which follows from the previous lemma to 
a subset of the interior of the domain. 

\begin{lemma} 
Let $D>1$ be as in Lemma $\ref{lem:stima}$. 
Then, there exists a constant $C>0$ such that for any $\sigma>0$ sufficiently small 
there exists $d_{\sigma}>0$ so large that
\be     \vert\nabla\phi^{d}\left(x\right)\vert  \leq   C   \qquad  \mathrm{in\ } \left\lbrace  x \in\overline{\Sigma}_{D}: 
           dist\left(x, D\Gamma_{D}\right)\geq \sigma \right\rbrace , \quad   d\geq d_{\sigma}.   
\label{stimaFi}
\ee
\label{lem:stima2}\end{lemma}

The next proposition is about the asymptotic behavior of the solutions 
of $(\ref{phi})$. 

\begin{lemma} 
Let $\phi^{d}$ be the solution of $(\ref{phi})$, then we have that
\be     \phi ^{d}\left(x\right) \rightarrow  \overline{\phi}\left(x\right):= \inf_{z\in\partial\Sigma_{D}}\left( \vert x-z\vert + \vert z-Q_{0}\vert\right),    \qquad \mathrm{as\ }   d\rightarrow\infty , \label{visc}\ee
uniformly on the compact sets of $\overline{\Sigma}_{D}$.
\label{lem:viscosity}\end{lemma} 

\begin{proof} 
We will show $(\ref{visc})$ in two steps: 
\begin{itemize}
\item[1)] we prove that the function on the right-hand side of $(\ref{visc})$ is the supremum of all the elements of 
\be       \emph{F}  = \lbrace v\in\emph{W}^{1, \infty}\left(\Sigma_{D}\right)    :     
          v\left(x\right) \leq \vert x- Q_{0}\vert       \mathrm{\ on\ } \partial\Sigma_{D},   \vert\nabla v\vert\leq 1  \mathrm{\ a.e.\ in\ }   \Sigma_{D}  \rbrace ;    \nonumber\ee
\item[2)] we prove that for any sequence $d_{k}\rightarrow\infty$, there is a subsequence $d_{k_{l}}\rightarrow\infty$ 
such that $\phi ^{d_{k_{l}}}\rightarrow\overline{\phi}$ uniformly 
on the compact sets of $\overline{\Sigma}_{D}$ as $d_{k_{l}}\rightarrow\infty$. 
Then it follows that $\phi ^{d}\rightarrow\overline{\phi}$ uniformly 
on the compact sets of $\overline{\Sigma}_{D}$ as $d\rightarrow\infty$.
\end{itemize}   
We first prove $1)$. To begin we show that $\overline{\phi}\in\emph{F}$. 
If $x_{1}, x_{2}\in\Sigma_{D}$ and $z_{2}\in\partial\Sigma_{D}$ realizes the infimum for $x_{2}$, we have 
\be     \vert \overline{\phi}\left(x_{1}\right) - \overline{\phi}\left(x_{2}\right)\vert  
             \leq \vert \vert x_{1} - z_{2}\vert + \vert z_{2} - Q_{0}\vert - \vert x_{2} - z_{2}\vert - \vert z_{2} - Q_{0}\vert \vert 
             \leq \vert x_{1} - x_{2} \vert .    
\nonumber\ee
Then, taking $x_{1}, x_{2}$ close, we get $\overline{\phi}\in\emph{W}^{1, \infty}\left(\Sigma_{D}\right)$ 
and $\vert\nabla\overline{\phi}\vert\leq 1$ a. e. in $\Sigma_{D}$. 
Moreover, it is easy to see that $\overline{\phi}\left(x\right)=\vert x-Q_{0}\vert$ if $x\in\partial\Sigma_{D}$. 
We next show that $\overline{\phi}$ is the maximum element of $\emph{F}$. 
We construct a $\delta$ neighborhood $\Sigma ^{\delta}_{D}$ of $\Sigma_{D}$ in this way: 
consider $Q_{0}=\left(-1, 0, \cdots, 0\right)$ and, for every $z\in\partial\Sigma_{D}$, 
the line from $Q_{0}$ to $z$. 
If $\delta>0$ is small enough, each point $x$ in $\Sigma ^{\delta}_{D}\setminus\Sigma_{D}$ 
is uniquely determined by the equation $x=z+\overline{\delta}r\left(z\right)$, 
where $z\in\partial\Sigma_{D}$ is the intersection point of the line from $Q_{0}$ to $x$ 
with $\partial\Sigma_{D}$, $r\left(z\right)$ is the unit outer vector on the line, 
and $0<\overline{\delta}<\frac{\delta}{\cos\theta\left(z\right)}$; 
here $\theta\left(z\right)$ is the angle between $r\left(z\right)$ and the unit outer normal at $z$, $\nu\left(z\right)$, 
in the plane generated by $r\left(z\right)$ and $\nu\left(z\right)$. 
Note that for the point on the boundary $z\in\lbrace z_{1}=z_{n}=0\rbrace$ 
we can consider $\nu\left(z\right)$ just taking the normal to the hypersurface defined by the equation $x_{1}\tan\alpha +x_{n}=0$ 
or to the one defined by the equation $x_{1}\tan\alpha -x_{n}=0$, 
and it is well defined since the angle $\theta\left(z\right)$ is the same for those points. 
In addition, the map $x\rightarrow\left(z, \overline{\delta}\right)$ is continuous 
in $\Sigma ^{\delta}_{D}\setminus\Sigma_{D}$. 

\noindent Now, we can extend every $v\in\emph{F}$ to a $\tilde{v}\in\emph{W}^{1, \infty}\left(\Sigma^{\delta}_{D}\right)$, 
taking $v=\tilde{v}$ in $\Sigma_{D}$ and $\tilde{v}\left(x\right)=v\left(z\right)$ for $x\in\Sigma^{\delta}_{D}\setminus\Sigma_{D}$. 
Moreover, if we consider the function
\bern 
 \tilde{K}\left(x\right)  =
 \left\{ 
\begin{array}{ll} 
           1 \qquad & \mathrm{in\ }   \Sigma_{D},    \\
           1 + C\overline{\delta} \qquad  & \mathrm{in\ }  \Sigma^{\delta}_{D}\setminus\Sigma_{D},
 \end{array} 
\right.         
 \nonumber\eern
for some large constant $C>0$ independent of $\delta$, we get $\vert\nabla\tilde{v}\vert\leq\tilde{K}$ a. e. in $\Sigma^{\delta}_{D}$. 
Now, we regularize $\tilde{v}$ using the convolution with mollifiers, 
that is considering, for $\lambda>0$ small enough, $v_{\lambda}:=\tilde{v}\ast\rho_{\lambda}$, 
with $\rho_{\lambda}\left(x\right)=\lambda^{-n}\rho\left(x/\lambda\right)$, 
$\rho\in\mathcal{C}^{\infty}_{0}\left(\R^{n}\right)$, $supp\rho\subset B_{1}\left(0\right)$, 
$\int_{\R^{n}}\rho\left(x\right)dx=1$. 
Then we have 
\be    \vert\nabla v_{\lambda}\vert \leq \vert\nabla\tilde{v}\vert\ast\rho_{\lambda} \leq  \tilde{K}\ast\rho_{\lambda} \leq   1  + C\lambda   \nonumber\ee
on $\Sigma_{D}$ and $v_{\lambda}\rightarrow v$ in $\mathcal{C}\left(\Sigma_{D}\right)$ as $\lambda\rightarrow 0$. 
Let now $x, y\in\Sigma_{D}$ and consider the function $\xi\left(t\right)=tx+\left(1-t\right)y$, for $t\in\left[0, 1\right]$; then we can estimate
\be
\vert  v_{\lambda}\left(x\right) - v_{\lambda}\left(y\right) \vert 
\leq \int_{0}^{1} \vert\nabla v_{\lambda}\left(\xi\left(t\right)\right)\vert \cdot \vert \frac{d\xi}{dt}\vert dt  
\leq \int_{0}^{1} \vert 1 + C\lambda\vert \cdot \vert x-y \vert dt 
\leq  \left(1 + C\lambda\right) \cdot \vert x-y \vert . 
\nonumber \ee
Letting $\lambda\rightarrow 0$, we obtain $\vert v\left(x\right)-v\left(y\right)\vert\leq\vert x-y\vert$. 
Hence $v\left(x\right)\leq v\left(y\right)+\vert x-y\vert$, and 
$v\left(x\right)\leq \vert y-Q_{0}\vert+\vert x-y\vert$ for all $y\in\partial\Sigma_{D}$. 
So $v\leq\overline{\phi}$. 

We next prove $2)$. 
By gradient estimate and the Ascoli-Arzel\`{a} theorem we know that the $\phi ^{d}$'s 
admit limit $\hat{\phi}$ in the whole closure of $\Sigma_{D}$. 
Moreover it is easy to see that $\hat{\phi}$ belong to the set $\emph{F}$; 
hence $\hat{\phi}\leq\overline{\phi}$. 
We need then to prove only $\overline{\phi}\leq\hat{\phi}$. 
Let $v\in\emph{F}$. Similarly to $1)$, we extend $v$ to $\tilde{v}$ in $\Sigma^{\delta}_{D}$ 
and regularize $\tilde{v}$ to $v_{\lambda}$ in such a way that we have 
$\norm{v-v_{\lambda}}_{L^{\infty}\left(\Sigma_{D}\right)}\leq C\lambda$ and $\vert\nabla\tilde{v}\vert\leq\tilde{K}$. 
Hence as before we get $\vert\nabla v_{\lambda}\vert\leq 1+C\lambda$ on $\Sigma_{D}$ and 
$v_{\lambda}\rightarrow v$ in $\mathcal{C}\left(\Sigma_{D}\right)$ as $\lambda\rightarrow 0$. 
By simple computation we obtain that $v_{\lambda}$ satisfies 
\bern
\left\{ 
\begin{array}{ll} 
         \frac{1}{d}\Delta  v_{\lambda}  - \vert\nabla v_{\lambda}  \vert ^{2} + 1 + C\lambda + \frac{1}{d} A_{\lambda} \geq  0   \quad & \mathrm{in\ } \Sigma_{D},      \\
        v_{\lambda} \leq  \vert x-Q_{0}\vert + C\lambda     \quad  &  \mathrm{on\ } \partial\Sigma_{D},   
 \end{array} 
\right.         
 \nonumber\eern 
where $A_{\lambda}\geq 0$. 
If we define 
\be       \tilde{v}_{\lambda}     :=     \frac{v_{\lambda}}{\sqrt{1 + C\lambda +\frac{1}{d} A_{\lambda}}}, \nonumber\ee
by comparison we deduce that
\be      \tilde{v}_{\lambda}  \leq  \phi ^{d\sqrt{1 + C\lambda +\frac{1}{d} A_{\lambda}}}   + C\lambda . \label{vtilde}\ee      
Choosing $d=d'_{k_{l}}$ in $(\ref{vtilde})$ such that 
\be    d_{k_{l}} = d'_{k_{l}}   \sqrt{1 + C\lambda +\frac{1}{d'_{k_{l}}} A_{\lambda}}, \nonumber\ee
we see that 
\be    \frac{v_{\lambda}}{\sqrt{1 + C\lambda}}  \leq     \hat{\phi}  +   C\lambda     \nonumber\ee
as $d'_{k_{l}}\rightarrow\infty$. 
Then, letting $\lambda\rightarrow 0$, we obtain $v\leq\hat{\phi}$; 
in particular, $\overline{\phi}\leq\hat{\phi}$. 
Hence $\overline{\phi}=\hat{\phi}$. 
\end{proof}

Next we analyze the asymptotic behavior of the solutions of $(\ref{logphi})$. 
From now on in this subsection we study separately the two cases 
$0<\alpha <\frac{\pi}{2}$ and $\frac{\pi}{2}\leq\alpha\leq\pi$. 
Let us consider the first case. 
\begin{proposition} 
Suppose that $0<\alpha <\frac{\pi}{2}$. 
Let $D$ be a large fixed constant and $\Phi^{d}$ the solution of $(\ref{logphi})$. 
Then we have 
\be     \Phi^{d}\left(x\right) \rightarrow \min \left\lbrace     d_{1}\left(x\right),  d_{2}\left(x\right)  \right\rbrace ,   \qquad \mathrm{as\ }   d\rightarrow\infty ,  \nonumber\ee
uniformly on the compact sets of $\overline{\Sigma}_{D}\cap\bar{B}_{\frac{D}{4}}\left(0\right)$, 
where
\bern
d_{1}\left(x\right) &:=& \sqrt{\left(x_{1} - \frac{\tan^{2}\alpha -1}{\tan^{2}\alpha + 1}\right)^{2}   +  \vert x'' \vert^{2}   +    \left(x_{n} - \frac{2\tan\alpha}{\tan^{2}\alpha + 1}\right)^{2}}, \label{d1}\\
d_{2}\left(x\right) &:=& \sqrt{\left(x_{1} - \frac{\tan^{2}\alpha -1}{\tan^{2}\alpha + 1}\right)^{2}   +  \vert x'' \vert^{2}   +    \left(x_{n} +   \frac{2\tan\alpha}{\tan^{2}\alpha + 1}\right)^{2}}. 
\label{d2}\eern  
\label{prop:asymptotic}
\end{proposition}

\begin{remark} 
Note that $d_{1}$ and $d_{2}$ are the distance functions, respectively, from the point $Q_{1}=\left(\frac{\tan^{2}\alpha -1}{\tan^{2}\alpha + 1}, 0, \cdots , 0, \frac{2\tan\alpha}{\tan^{2}\alpha + 1}\right)$, which is the symmetrical point to $Q_{0}$ with respect to the hypersurface defined by the equation $x_{1}\tan\alpha +x_{n}=0$, and from the point $Q_{2}=\left(\frac{\tan^{2}\alpha -1}{\tan^{2}\alpha + 1}, 0, \cdots , 0, -\frac{2\tan\alpha}{\tan^{2}\alpha + 1}\right)$, which is the symmetrical point to $Q_{0}$ with respect to the hypersurface defined by the equation $x_{1}\tan\alpha -x_{n}=0$. 
So the function $\overline{\phi}\left(x\right)$ is even with respect to the coordinate $x_{n}$ and a.e. differentiable. 
The problem is that it does not have zero $x_{n}$-derivative on $\left\lbrace x_{n}=0\right\rbrace$.   
\end{remark}

\begin{proof} 
If $\phi^{d}$ is a solution of $(\ref{phi})$, it is easy to see that 
$\phi^{d}+\sup_{x\in\partial\Sigma_{D}}\vert\vert x-Q_{0}\vert +\frac{1}{d}\log\left(U\left(d\left(x -Q_{0}\right)\right)\right)\vert$ 
is a supersolution of $(\ref{logphi})$ and 
$\phi^{d}-\sup_{x\in\partial\Sigma_{D}}\vert\vert x-Q_{0}\vert +\frac{1}{d}\log\left(U\left(d\left(x -Q_{0}\right)\right)\right)\vert$ 
is a subsolution. 
Then $\Phi^{d}$ must lie in between these two functions. 
Hence, by Lemma $\ref{lem:bordo}$, it is sufficient to prove the analogous statement for $\phi^{d}$. 
The proof of the latter fact is a consequence of Lemma $\ref{lem:viscosity}$ and 
the following Lemma $\ref{lem:viscosity2}$.  
\end{proof}

\begin{lemma} 
Suppose that $0<\alpha <\frac{\pi}{2}$. 
If $\overline{\phi}\left(x\right)$ is as in $(\ref{visc})$, then 
\be     
\overline{\phi}\left(x\right) = \min \left\lbrace     d_{1}\left(x\right),  d_{2}\left(x\right)  \right\rbrace   ,     \qquad    x\in \bar{B}_{\frac{D}{4}}\left(0\right),  \nonumber \ee
where $d_{1}$ and $d_{2}$ are as in $(\ref{d1})$ and $(\ref{d2})$. 
\label{lem:viscosity2}\end{lemma}

\begin{proof} 
Consider a point $x=\left(x_{1}, \cdots , x_{n}\right)$ with $x_{n}\geq 0$. 
By construction of $\Sigma_{D}$, the point $z\in\partial\Sigma_{D}$ 
which realizes the infimum will necessarily belong to the set 
$\left\lbrace\left\lbrace x_{1}\tan\alpha +x_{n}=0\right\rbrace\cap\left\lbrace x_{1}<0\right\rbrace\right\rbrace$. 
This implies that 
\be   \overline{\phi}\left(x\right)   =   \inf_{z\in\left\lbrace\left\lbrace x_{1}\tan\alpha +x_{n}=0\right\rbrace\cap\left\lbrace x_{1}<0\right\rbrace\right\rbrace}  \left(\vert x-z\vert + \vert z-Q_{0}\vert\right).      \nonumber\ee
Now we can reason as follows: given $x$, the level sets of the function $z\rightarrow\vert x-z\vert +\vert z-Q_{0}\vert$ 
are the axially symmetric ellipsoids with focal points $x$ and $Q_{0}$. 
The smaller is the ellipsoid, the smaller is the value of this function; 
so we are reduced to find the smallest ellipsoid which intersects 
$\left\lbrace\left\lbrace x_{1}\tan\alpha +x_{n}=0\right\rbrace\cap\left\lbrace x_{1}<0\right\rbrace\right\rbrace$. 
We note that if we fix $x_{1}$, $x_{n}$ and vary only $x''$, the corresponding infimum $z$ 
has the same $z_{1}$, $z_{n}$ and different $z''$; 
so we can determine $z_{1}$, $z_{n}$ in the simplest case $x''=\left(0, \cdots , 0\right)$, 
and obviously $z''=\left(0, \cdots , 0\right)$. 
Then we are reduced to consider the minimum problem 
\be  
\min_{\left(z_{1}, z_{n}\right)\in 
\left\lbrace\left\lbrace x_{1}\tan\alpha +x_{n}=0\right\rbrace\cap\left\lbrace x_{1}<0\right\rbrace\right\rbrace}
\left(\sqrt{\left(x_{1}-z_{1}\right)^{2}  +   \left(x_{n} + \tan\alpha z_{1}\right)^{2}}    +       \sqrt{\left(z_{1}+1\right)^{2}  + \tan^{2}\alpha z_{1}^{2}}    \right).     
 \nonumber  \ee 
Deriving with respect to the variable $z_{1}$ we obtain that at a minimum point 
 \be     \frac{-\left(x_{1}-z_{1}\right)+\tan\alpha\left(x_{n}+\tan\alpha z_{1}\right)}{\sqrt{\left(x_{1}-z_{1}\right)^{2}  +   \left(x_{n} + \tan\alpha z_{1}\right)^{2}}}    +    \frac{\left(z_{1}+1\right)+\tan^{2}\alpha z_{1}}{\sqrt{\left(z_{1}+1\right)^{2}  + \tan^{2}\alpha z_{1}^{2}}} =0,   \nonumber\ee 
which implies 
\bern 
 z_{1} = \frac{-2\tan\alpha x_{1} + \left(\tan^{2}\alpha -1\right)x_{n}}{\left(\tan^{2}\alpha +1\right)\left(\tan\alpha x_{1} + x_{n} - \tan\alpha\right)}, \label{z1}\\
 z_{n} =  \frac{2\tan^{2}\alpha x_{1} - \tan\alpha\left(\tan^{2}\alpha -1\right)x_{n}}{\left(\tan^{2}\alpha +1\right)\left(\tan\alpha x_{1} + x_{n} - \tan\alpha\right)}.  \label{zn}  
\eern 
Now assume that $x''\neq\left(0, \cdots , 0\right)$ and $x_{1}$, $x_{n}$ are as before. 
By the previous observation we know that the coordinates $z_{1}$, $z_{n}$ of the corresponding infimum 
are given by $(\ref{z1})$ and $(\ref{zn})$. 
So we have to determine only $z''$. 
To do this let us consider the minimum problem
\be   \min_{z''\in\R^{n-2}}
\left(\sqrt{\left(x_{1}-z_{1}\right)^{2}  + \vert x''-z''\vert^{2}  +  \left(x_{n} + \tan\alpha z_{1}\right)^{2}}    +       \sqrt{\left(z_{1}+1\right)^{2}  +   \vert z''\vert^{2}   +  \tan^{2}\alpha z_{1}^{2}} \right).     
 \label{minprob2}  \ee    
Again by differentiation we obtain that a minimum point must satisfy
\be     \frac{z''-x''}{\sqrt{\left(x_{1}-z_{1}\right)^{2}  + \vert x''-z''\vert^{2}  +  \left(x_{n} + \tan\alpha z_{1}\right)^{2}}}     +     
          \frac{z''}{\sqrt{\left(z_{1}+1\right)^{2}  +   \vert z''\vert^{2}   +  \tan^{2}\alpha z_{1}^{2}}}    =   0,     \nonumber\ee
which gives
\be      z''  =    x''  \frac{\sqrt{\left(z_{1}+1\right)^{2}   +  \tan^{2}\alpha z_{1}^{2}}}{\sqrt{\left(x_{1}-z_{1}\right)^{2}  +   \left(x_{n} + \tan\alpha z_{1}\right)^{2}} +    \sqrt{\left(z_{1}+1\right)^{2}   +  \tan^{2}\alpha z_{1}^{2}}}.     \label{z''} \ee
If we plug $(\ref{z1})$, $(\ref{zn})$ and $(\ref{z''})$ into $(\ref{minprob2})$, 
we obtain that $\overline{\phi}\left(x\right)= d_{1}\left(x\right)$. 
Reasoning in the same way for points with $x_{n}<0$, we have $\overline{\phi}\left(x\right)= d_{2}\left(x\right)$. 
Then we get the conclusion.
\end{proof}

\begin{remark} 
Note that $\overline{\phi}$ is a viscosity solution of the Hamilton-Jacobi equation 
$\vert\nabla\phi\vert^{2}=1$ in $\Sigma_{D}$. 
In fact, what we have to show is that 
\begin{itemize}
\item[i)] $\vert p\vert^{2}\leq 1$, for every $x\in\Sigma_{D}$ and every $p\in D^{+}\overline{\phi}\left(x\right)$,
\item[ii)] $\vert p\vert^{2}\geq 1$, for every $x\in\Sigma_{D}$ and every $p\in D^{-}\overline{\phi}\left(x\right)$,
\end{itemize}
where $D^{+}\overline{\phi}\left(x\right)$ and $D^{-}\overline{\phi}\left(x\right)$ are 
respectively the superdifferential and the subdifferential of $\overline{\phi}$ at $x$. 
Now we can use the description of $D^{+}\overline{\phi}\left(x\right)$ and $D^{-}\overline{\phi}\left(x\right)$ 
given in Theorem $3.4.4$ in \cite{CS}: 
let $\Omega\subset\R^{n}$ be open and $S\subset\R^{m}$ be compact; 
let $F=F\left(s, x\right)$ be continuous in $S\times\Omega$ together with its partial derivative $D_{x}F$, 
and let us define $u\left(x\right)=\min_{s\in S}F\left(s, x\right)$; 
given $x\in\Omega$, let us set 
\be     M\left(x\right)=\left\lbrace s\in S : u\left(x\right) =  F\left(s, x\right)\right\rbrace ,  \qquad    
           Y\left(x\right)=\left\lbrace  D_{x}F\left(s, x\right)  :     s\in   M\left(x\right)\right\rbrace .      \nonumber  \ee
Then, for any $x\in\Omega$, 
\be     D^{+}u\left(x\right) = co\left(Y\left(x\right)\right),     \label{D+}\ee
and 
\bern 
  D^{-}u\left(x\right) =     
\left\{ 
\begin{array}{ll} 
           \left\lbrace p\right\rbrace  & \mathrm{if\ }   Y\left(x\right) = {p},    \\
          \emptyset   & \mathrm{if\ }  Y\left(x\right)  \mathrm{\ is\ not\ a\ singleton}.
 \end{array} 
\right.         
 \label{D-}\eern
\noindent Now we can take $\Omega=\Sigma_{D}$, $S=\left\lbrace Q_{1}, Q_{2}\right\rbrace$ 
and $\overline{\phi}\left(x\right)=\min_{i\in\left\lbrace 1, 2\right\rbrace}\left\lbrace d_{i}\left(x\right)\right\rbrace$; 
so 
\be    M\left(x\right)=\left\lbrace Q_{i}:\overline{\phi}\left(x\right) =  d_{i}\left(x\right)\right\rbrace ,  \qquad
          Y\left(x\right)=\left\lbrace  D_{x}d_{i}\left(x\right):Q_{i}\in M\left(x\right)\right\rbrace .  \nonumber\ee
Then, using $(\ref{D+})$ and $(\ref{D-})$, it is easy to see that, 
if we take $x\in\Sigma_{D}$ with $x_{n}>0$, 
then $D^{+}\overline{\phi}\left(x\right)=D^{-}\overline{\phi}\left(x\right)=\left\lbrace D_{x}d_{1}\left(x\right)\right\rbrace$; 
in the same way, if $x_{n}<0$, then 
$D^{+}\overline{\phi}\left(x\right)=D^{-}\overline{\phi}\left(x\right)=\left\lbrace D_{x}d_{2}\left(x\right)\right\rbrace$. 
So in these two cases properties $i)$, $ii)$ are trivially verified. 
In the case $x_{n}=0$, we have that $\overline{\phi}\left(x\right)=d_{1}\left(x\right)=d_{2}\left(x\right)$; 
then $M\left(x\right)=\left\lbrace Q_{1}, Q_{2}\right\rbrace$ and 
$Y\left(x\right)=\left\lbrace D_{x}d_{1}\left(x\right), D_{x}d_{2}\left(x\right)\right\rbrace$. 
Hence, using again $(\ref{D+})$, $(\ref{D-})$, we obtain 
$D^{+}\overline{\phi}\left(x\right)=co\left\lbrace\frac{x-Q_{1}}{d_{1}\left(x\right)}, \frac{x-Q_{2}}{d_{2}\left(x\right)}\right\rbrace = \frac{x-co\left\lbrace Q_{1}, Q_{2}\right\rbrace}{\overline{\phi}\left(x\right)}$ and $D^{-}\overline{\phi}\left(x\right)=\emptyset$. 
Then we have only to prove property $i)$, since $ii)$ is again trivially verified. 
To show $i)$ it is sufficient to observe that every $p\in D^{+}\overline{\phi}\left(x\right)$ 
is of the form $p=\frac{x-Q}{\overline{\phi}\left(x\right)}$, where $Q$ belongs to the line joining $Q_{1}$ to $Q_{2}$, 
and that $\vert x-Q\vert\leq\overline{\phi}\left(x\right)$. 
\end{remark}

\medskip

Let us consider now the case $\frac{\pi}{2}\leq\alpha\leq\pi$. 
We have the analogous of the Proposition $\ref{prop:asymptotic}$. 
\begin{proposition} 
Suppose that $\frac{\pi}{2}\leq\alpha\leq\pi$. 
Let $D$ be a large fixed constant and $\Phi^{d}$ the solution of $(\ref{logphi})$. 
Then we have 
\be     \Phi^{d}\left(x\right) \rightarrow \bar{\Phi}\left(x\right)  ,   \qquad \mathrm{as\ }   d\rightarrow\infty ,  \label{converg}\ee
uniformly on the compact sets of $\overline{\Sigma}_{D}\cap\bar{B}_{\frac{D}{4}}\left(0\right)$, 
where
\bern
    \bar{\Phi}\left(x\right) = 
     \left\{ 
\begin{array}{ll} 
         \min \left\lbrace d_{1}\left(x\right), d_{2}\left(x\right)\right\rbrace , \quad & \mathrm{if\ }    
         \tan\alpha\leq\frac{x_{1}-\sqrt{x^{2}_{1} + x^{2}_{n}}}{x_{n}},      \\
        \sqrt{\left(1 + \sqrt{x^{2}_{1} + x^{2}_{n}}\right)^{2} +\vert x''\vert^{2}},   \quad     &    \mathrm{if\ } 
        \tan\alpha\geq\frac{x_{1}-\sqrt{x^{2}_{1} + x^{2}_{n}}}{x_{n}}.  
\end{array} 
\right.         
\label{caso2}\eern  
\label{prop:asymptotic2}
\end{proposition}

\begin{proof} 
We can reason as in the proof of Proposition $\ref{prop:asymptotic}$, 
obtaining that it is sufficient to show the convergence in $(\ref{converg})$ 
for the function $\phi^{d}$. 
To prove the latter assertion we have to use Lemma $\ref{lem:viscosity}$, 
together with the fact that in the case $\frac{\pi}{2}\leq\alpha\leq\pi$ 
the function $\overline{\phi}$ defined in $(\ref{visc})$ 
is equal to that one defined in $(\ref{caso2})$. 
We can obtain this expression by mixing the arguments used in 
the proof of Lemma 
$\ref{lem:viscosity2}$ and those used in Lemma $3.9$ in \cite{GMMP1}. 
\end{proof}

\subsection{Case  $\pi<\alpha <2\pi$} \label{sec:3.2}
In this case we construct the domain $\Sigma$ in the following way: 
we consider the set $\left\lbrace x_{n}=0\right\rbrace\cap\left\lbrace x_{1}\leq 0\right\rbrace$ 
and the hypersurface defined by the equation $x_{1}\tan\alpha +x_{n}=0$ with $x_{n}\leq 0$. 
Then we close the domain with a smooth surface; 
the following figure represents a section of the domain in the plane $x_{1}$, $x_{n}$.
\begin{center}
\includegraphics[scale=0.5]{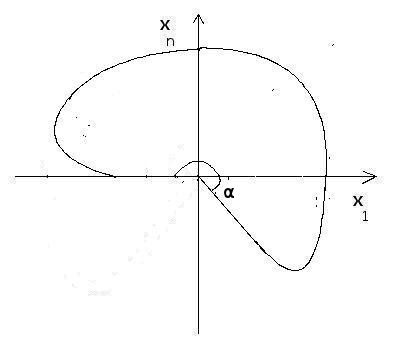} 
\end{center}
We define the scaled domain $\Sigma_{D}$ as in $(\ref{domain})$ and 
denote by $\Gamma_{D}$ the singularity, which lies on $\left\lbrace x_{1}=x_{n}=0\right\rbrace$. 
As in the previous case, the solution of a Dirichlet problem in $\Sigma_{D}$ 
will be qualitatively similar to that of $(\ref{dirichlet})$. 

We have to study the asymptotic behavior of the solution of the problem
\bern
\left\{ 
\begin{array}{ll} 
         -\frac{1}{d^{2}}\Delta \varphi + \varphi = 0   \quad & \mathrm{in\ } \Sigma_{D},      \\
           \varphi = U\left(d\left(\cdot -Q_{0}\right)\right)  \quad     &    \mathrm{on\ } \partial\Sigma_{D},  
\end{array} 
\right.         
 \nonumber\eern  
To do this we consider the function $\phi=-\frac{1}{d}\log\varphi$, which satisfies 
\bern
\left\{ 
\begin{array}{ll} 
         \frac{1}{d}\Delta\phi - \vert\nabla\phi\vert ^{2} + 1 = 0   \quad & \mathrm{in\ } \Sigma_{D},      \\
        \phi =  -\frac{1}{d}\log\left(U\left(d\left(\cdot -Q_{0}\right) \right) \right)  \quad  &  \mathrm{on\ } \partial\Sigma_{D}.    
 \end{array} 
\right.         
 \label{logphi2}\eern  
Since the asymptotic analysis is very similar to that one made in Subsection $\ref{sec:asympt}$ for $0<\alpha\leq\pi$ 
we will not repeat the computations. 
What we obtain is the following result:
\begin{proposition} 
Suppose that $\pi <\alpha <2\pi$. 
Let $D$ be a large fixed constant and $\Phi^{d}$ the solution of $(\ref{logphi2})$. 
Then we have 
\be     \Phi^{d}\left(x\right) \rightarrow dist\left(x, Q_{0}\right) = 
          \sqrt{\left(x_{1} + 1\right)^{2} + \vert x'\vert^{2}},   \qquad \mathrm{as\ }   d\rightarrow\infty ,  \nonumber \ee
uniformly on the compact sets of $\overline{\Sigma}_{D}\cap\bar{B}_{\frac{D}{4}}\left(0\right)$. 
\label{prop:asymptotic3}
\end{proposition}

\subsection{Definition of the approximate solutions}  \label{sec:3.3}
In order to apply the theory in Subsection $\ref{perturbation}$, 
in this subsection we construct a manifold of approximate solutions to $(\ref{problem1})$. 
Since the limit function of the solutions of $(\ref{logphi})$ is not the same for different angles $\alpha$, 
as we have seen in Subsections $\ref{sec:3.1}$ and $\ref{sec:3.2}$, we will distinguish the cases. 
We will give the precise construction only for $0<\alpha <\frac{\pi}{2}$; 
in fact in this case the computations are quite different from the flat case $\alpha =\pi$. 
In the other cases the estimates for the approximate solutions are the same (for $\frac{\pi}{2}\leq\alpha\leq\pi$) 
or very similar (for $\pi <\alpha <2\pi$) to that ones obtained in \cite{GMMP1}, Subsection $3.2$, 
and then we will omit the proofs.

\subsubsection{Case $0<\alpha <\frac{\pi}{2}$}  \label{case 1}
Since the function $\bar{u}_{\epsilon, Q}$ defined in Subsection $\ref{apprNeumann}$ 
is an approximate solution of $(\ref{problem1})$ with pure Neumann boundary conditions, 
we need to modify it in the following way. 
If $\Phi^{d}$ the solution of $(\ref{logphi})$, the function 
\be     \Xi_{d}\left(y\right)   =    e^{-d\Phi^{d}\left(\frac{y}{d}+Q_{0}\right)}      \label{Xi}\ee
solves the problem 
\bern
\left\{ 
\begin{array}{ll} 
         -\Delta \Xi_{d} + \Xi_{d} = 0   \quad & \mathrm{in\ }    d\left(\Sigma_{D}-Q_{0}\right),      \\
           \Xi_{d} = U\left(\cdot\right)  \quad     &    \mathrm{on\ } d\left(\partial\Sigma_{D}-Q_{0}\right).  
\end{array} 
\right.         
 \label{Xiprob}\eern
We can obtain a solution to $(\ref{Xiprob})$ looking at the minimum problem
 \be      \inf_{v=U \mathrm{\ on\ } d\left(\partial\Sigma_{D}-Q_{0}\right)} 
 \left\lbrace \int_{d\left(\Sigma_{D}-Q_{0}\right)}\left(\vert\nabla v\vert^ {2} + v^ {2}\right)dy \right\rbrace . \label{Xiprobmin}\ee

From $(\ref{Xiprobmin})$ we can derive norm estimate on $\Xi_{d}$. 
In fact, we can take a cut-off function $\chi_{1}:d\left(\overline{\Sigma}_{D}-Q_{0}\right)\rightarrow\R$ such that 
\bern
\left\{ 
\begin{array}{ll} 
           \chi_{1}\left(y\right) = 1 \qquad & \mathrm{for\ }   dist\left(y, d\left(\partial\Sigma_{D}-Q_{0}\right)\right)\leq\frac{1}{2},    \\
            \chi_{1}\left(y\right) = 0 \qquad  & \mathrm{for\ } y\in d\left(\Sigma_{D}-Q_{0}\right),  \mathrm{\ }   dist\left(y, d\left(\partial\Sigma_{D}-Q_{0}\right)\right)\geq 1, \\
            \vert\nabla\chi_{1}\left(y\right) \vert \leq 4   \qquad  & \mathrm{for\ all\ } y,
 \end{array} 
\right.         
 \nonumber\eern 
and then consider the function $\bar{v}\left(y\right)=\chi_{1}\left(y\right)U\left(y\right)$. 
It is easy to see that $\norm{\bar{v}}_{H^{1}\left(d\left(\Sigma_{D}-Q_{0}\right)\right)}\leq e^{-d\left(1+o\left(1\right)\right)}$, 
so by $(\ref{Xiprobmin})$ we find that
\be             \norm{\Xi_{d}}_{H^{1}\left(d\left(\Sigma_{D}-Q_{0}\right)\right)}\leq \norm{\bar{v}}_{H^{1}\left(d\left(\Sigma_{D}-Q_{0}\right)\right)}\leq e^{-d\left(1+o\left(1\right)\right)}.       \label{Xinorm}\ee
We can also obtain pointwise estimates on $\Xi_{d}$. In fact, from Proposition $\ref{prop:asymptotic}$ we obtain that, 
as $d\rightarrow+\infty$, 
\be      \Xi_{d}\left(y\right) =    \exp \left[-\min \left\lbrace \sqrt{\left(y_{1} - d - \frac{d\left(\tan^{2}\alpha -1\right)}{\tan^{2}\alpha + 1}\right)^{2}   +  \vert y'' \vert^{2}   +    \left(y_{n} \mp \frac{2d\tan\alpha}{\tan^{2}\alpha + 1}\right)^{2}} \right\rbrace \right]  \cdot  e^{o\left(d\right)}, 
\label{Xipoint}\ee
for $y\in d\left(V-Q_{0}\right)$, where $V$ is any set compactly contained in $\overline{\Sigma}_{D}$. 
Finally, we have pointwise estimates for the gradient of $\Xi_{d}$. 
Indeed, using the uniform convergence in $(\ref{bordo})$ and reasoning as in the proof of Lemmas 
$\ref{lem:stima1}$ and $\ref{lem:stima2}$, we obtain that $(\ref{stimaFi})$ holds true also for $\Phi_{d}$. 
Then we can apply the arguments in \cite{LN} 
(see in particular Proposition $1.4$, Lemma $1.5$ and Lemma $B.1$) 
to conclude that $\nabla\Phi_{d}\rightarrow\nabla\overline{\phi}$ uniformly as $d\rightarrow+\infty$ 
in any set compactly contained in $\overline{\Sigma}_{D}$ on which $\nabla\overline{\phi}$ is defined. 
This convergence implies that, as $d\rightarrow+\infty$, 
\bern       \nabla\Xi_{d}\left(y\right)  =   - \exp \left[-\min \left\lbrace \sqrt{\left(y_{1} - d - \frac{d\left(\tan^{2}\alpha -1\right)}{\tan^{2}\alpha + 1}\right)^{2}   +  \vert y'' \vert^{2}   +    \left(y_{n} \mp \frac{2d\tan\alpha}{\tan^{2}\alpha + 1}\right)^{2}} \right\rbrace \right]    \nonumber\\
       \cdot  e^{o\left(d\right)}      
   \cdot \left(\nabla\overline{\phi}\left(\frac{y}{d}+Q_{0}\right)  + o\left(1\right)\right),   \label{Xigradpoint}\eern
for $y\in d\left(V-Q_{0}\right)$, where $V$ is as before. 

Now, we want to obtain similar bounds and estimates 
for $\frac{\partial\Xi_{d}}{\partial d}$ and its gradient. 
Using the definition of $\Xi_{d}\left(y\right)=\varphi\left(\frac{y}{d}+Q_{0}\right)$ 
and the fact that also $\varphi$ depends on $d$, we have that 
\be   \frac{\partial\Xi_{d}}{\partial d}\left(y\right) = 
        \frac{\partial\varphi}{\partial d}\left(\frac{y}{d}+Q_{0}\right)  -
        \frac{y}{d^{2}}\cdot\nabla\varphi\left(\frac{y}{d}+Q_{0}\right).  \ee
Since $\varphi$ is the solution of $(\ref{varphi})$, we can differentiate 
$(\ref{varphi})$ obtaining 
\bern
\left\{ 
\begin{array}{ll} 
         -\frac{1}{d^{2}}\Delta\frac{\partial\varphi}{\partial d} +      
          \frac{\partial\varphi}{\partial d} =  -\frac{2}{d^{3}}\Delta\varphi  
             = -\frac{2}{d}\varphi      \quad & \mathrm{in\ } \Sigma_{D},      \\
     \frac{\partial\varphi}{\partial d}\left(x\right)  
           = \nabla U\left(d\left(x -Q_{0}\right)\right)\cdot \left(x-Q_{0}\right)   \quad     &    \mathrm{on\ } \partial\Sigma_{D},  
\end{array} 
\right.         
 \label{varphi1}\eern 
Because of the asymptotic behavior of $U$ at infinity, 
there exists a positive constant $C_{D}$ such that 
for $d$ large we have
\be \frac{1}{C_{D}} U\left(d\left(x-Q_{0}\right)\right) \leq 
     -\nabla U\left(d\left(x-Q_{0}\right)\right)\cdot \left(x-Q_{0}\right) \leq 
     C_{D}  U\left(d\left(x-Q_{0}\right)\right).     \label{CD}\ee
Hence, from $(\ref{varphi})$, $(\ref{CD})$, the fact that $\varphi>0$ 
and the maximum principle we obtain that 
$\varsigma :=-\frac{\partial\varphi}{\partial d}\geq\frac{1}{C_{D}}\varphi$ 
in $\widehat{\Gamma}_{D}$. 
Moreover, as for ($\ref{logphi})$ we can check that the function 
$\Upsilon^{d}:=-\frac{1}{d}\log\varsigma$ satisfies 
\bern
\left\{ 
\begin{array}{ll} 
         \frac{1}{d}\Delta\Upsilon ^{d} + \vert\nabla\Upsilon ^{d}\vert ^{2} 
          + 1  -  \frac{\varphi}{d\varsigma}  =  0 
            \quad & \mathrm{in\ } \Sigma_{D},      \\
          \Upsilon ^{d}\left(x\right)  
           = -\frac{1}{d}\log\left(-\nabla U\left(d\left(x -Q_{0}\right)\right) \cdot \left(x-Q_{0}\right)\right)    \quad     &    \mathrm{on\ } \partial\Sigma_{D},  
\end{array} 
\right.         
 \label{Upsilon}\eern
Since $\frac{\varphi}{\varsigma}$ stays bounded, 
$\frac{\varphi}{d\varsigma}$ tends to zero as $d\rightarrow +\infty$. 
Moreover, using again the asymptotic behavior of $U$ at infinity, 
we can say that the boundary datum in $(\ref{Upsilon})$ converges 
in every smooth sense (where $\partial\Sigma_{D}$ is regular) 
to $\vert x-Q_{0}\vert$ as $d\rightarrow +\infty$. 
As a consequence, the previous analysis adapts to $\Upsilon ^{d}$ 
and allows to conclude that still 
\be  \Upsilon ^{d}\rightarrow\overline{\phi}   \quad  
    \mathrm{and}   \quad  \nabla\Upsilon ^{d}\rightarrow\nabla\overline{\phi} 
\label{Upsilonlimit}\ee
uniformly as $d\rightarrow +\infty$ in any set compactly contained 
in $\overline{\Sigma}_{D}$ on which $\nabla\overline{\phi}$ is defined.   
  
From $(\ref{varphi1})$, reasoning as for $(\ref{Xinorm})$, we have that
\be  \norm{\frac{\partial\varphi}{\partial d}\left(\frac{\cdot}{d} + Q_{0}\right)}_{H^{1}\left(d\left(\Sigma_{D}-Q_{0}\right)\right)} 
    \leq e^{-d\left(1+o\left(1\right)\right)}.       \label{varnorm}\ee 
On the other hand, from $(\ref{Xiprob})$ one finds that the function 
$\varpi :=\frac{y}{d^{2}}\cdot\nabla\varphi\left(\frac{y}{d}+Q_{0}\right)=\frac{y}{d}\cdot\nabla\Xi_{d}\left(y\right)$ satisfies 
\be    -\Delta\varpi + \varpi = -\frac{2}{d}\Xi_{d}  \qquad \mathrm{in\ }    d\left(\Sigma_{D}-Q_{0}\right).   \nonumber\ee 
To control the boundary value of $\varpi$ we divide 
$\partial d\left(\Sigma_{D}-Q_{0}\right)$ into its 
intersection with $\left\lbrace y_{n}=0\right\rbrace$ and its complement. 
In the first region we have simply that $\varpi=\frac{y}{d}\cdot\nabla U\left(y\right)$. 
In the second instead the estimates in $(\ref{Xipoint})$ and $(\ref{Xigradpoint})$ 
hold true, which shows that the $L^{2}$ norm of the trace of $\varpi$ 
on $\partial d\left(\Sigma_{D}-Q_{0}\right)$ is of order 
$e^{-d\left(1+o\left(1\right)\right)}+e^{-d\left[1+\frac{2\tan\alpha}{\sqrt{\tan^{2}\alpha+1}}\right]\left(1+o\left(1\right)\right)}$. 
This fact and the latter formula imply that 
\be  \norm{\varpi}_{H^{1}\left(d\left(\Sigma_{D}-Q_{0}\right)\right)} 
    \leq e^{-d\left(1+o\left(1\right)\right)}  + 
       e^{-d\left[1+\frac{2\tan\alpha}{\sqrt{\tan^{2}\alpha+1}}\right]\left(1+o\left(1\right)\right)}. \label{varnorm1}\ee 
Then, from $(\ref{varnorm})$ and $(\ref{varnorm1})$, we conclude that 
\be   \norm{\frac{\partial\Xi_{d}}{\partial d}}_{H^{1}\left(d\left(\Sigma_{D}-Q_{0}\right)\right)} 
    \leq e^{-d\left(1+o\left(1\right)\right)}  + 
        e^{-d\left[1+\frac{2\tan\alpha}{\sqrt{\tan^{2}\alpha+1}}\right]\left(1+o\left(1\right)\right)}. \label{Xidpoint}\ee 
Now, using the fact that $\varphi\leq C_{D}\vert\frac{\partial\varphi}{\partial d}\vert$ 
and $(\ref{Upsilonlimit})$, together with the Harnack inequality         
(which implies $\vert\nabla\varphi\vert\leq Cd\varphi$ in $d\left(V-Q_{0}\right)$) 
one also finds 
\bern    \frac{\partial\Xi_{d}}{\partial d}\left(y\right) =    
 -\exp \left[-\min \left\lbrace \sqrt{\left(y_{1} - \frac{2d\tan^{2}\alpha}{\tan^{2}\alpha + 1}\right)^{2}   +  \vert y'' \vert ^{2}   +    \left(y_{n} \mp \frac{2d\tan\alpha}{\tan^{2}\alpha + 1}\right)^{2}} \right\rbrace \right]  \nonumber\\
 \cdot e^{o\left(d\right)} \cdot \left(1 + o\left(\frac{\vert y\vert}{d}\right)\right), 
\label{dpoint}\eern
and 
\be   \vert \nabla \frac{\partial\Xi_{d}}{\partial d}\left(y\right)\vert  \leq    
 \exp \left[-\min \left\lbrace \sqrt{\left(y_{1} - \frac{2d\tan^{2}\alpha}{\tan^{2}\alpha + 1}\right)^{2}   +  \vert y'' \vert^{2}   +    \left(y_{n} \mp \frac{2d\tan\alpha}{\tan^{2}\alpha + 1}\right)^{2}} \right\rbrace \right]  \cdot e^{o\left(d\right)}, 
\label{dgradpoint}\ee
for $d\left(V-Q_{0}\right)$ and $d\rightarrow +\infty$.

\bigskip
 
After these preliminaries, we are now in position to introduce our approximate solutions. 
Let us define two smooth non negative cut-off functions $\chi_{D}:\R^{n}\rightarrow\R$, 
$\chi_{0}:\R\rightarrow\R$ satisfying respectively 
\bern
\left\{ 
\begin{array}{ll} 
           \chi_{D}\left(y\right) = 1 \qquad & \mathrm{for\ }   \vert y\vert\leq\frac{dD}{16},    \\
            \chi_{D}\left(y\right) = 0 \qquad  & \mathrm{for\ } \vert y\vert\geq\frac{dD}{8}, \\
            \vert\nabla\chi_{D}\vert \leq  \frac{32}{dD}    \qquad  & \mathrm{on\ }\R^{n},
 \end{array} 
\right.         
\label{chiD}\eern 
and 
\bern
\left\{ 
\begin{array}{ll} 
           \chi_{0}\left(y\right) = 1 \qquad & \mathrm{for\ }    y\leq 0,    \\
            \chi_{0}\left(y\right) = 0 \qquad  & \mathrm{for\ }  y\geq 1, \\
             \chi_{0}  \mathrm{\ is\ non\ increasing\ } \qquad  & \mathrm{on\ }\R .
 \end{array} 
\right.         
 \label{chi0}\eern 
Now, using the new coordinates $y$ in Subsection $\ref{apprNeumann}$, we define
\be      u_{\epsilon, Q}\left(y\right)  :=  \chi_{\mu_{0}}\left(\epsilon y\right)\left[ \left(U_{Q}\left(y\right) - \Xi_{d}\left(y\right)\right) \chi_{D}\left(y\right) +\epsilon w_{Q}\left(y\right) \chi_{0}\left(y_{1}-d\right) \right].     \label{apprfunc}\ee 
Following the line of \cite{GMMP1} 
we prove that the $u_{\epsilon, Q}$'s are good approximate solutions to $(\ref{problem1})$ for suitable conditions of $Q$.

\begin{proposition} 
Let $\mu_{0}$ be the constant appearing in Subsection $\ref{apprNeumann}$. 
Then there exists another constant $C_{\Omega}$, independent of $\epsilon$, 
such that, for $C_{\Omega}\leq d\leq\frac{1}{\epsilon C_{\Omega}}$ and for $Dd<\frac{\mu_{0}}{\epsilon C_{\Omega}}$, 
the functions $u_{\epsilon, Q}$ satisfy 
\bern     \norm{I'_{\epsilon}\left(u_{\epsilon, Q}\right)}  &\leq &  C\left(\epsilon^{2}  + \epsilon e^{-d\left(1+o\left(1\right)\right)}  
        + e^{-d\left[\frac{1}{2}\sqrt{\frac{D\tan\alpha\left(\tan\alpha +1\right)}{\tan^{2}\alpha +1}} + \frac{2\tan\alpha}{\sqrt{\tan^{2}\alpha +1}}\right]    \left(1+o\left(1\right)\right)}\right)   \nonumber\\
             &+& C\left(e^{-\frac{d\left(p+1\right)}{2}\left(1+o\left(1\right)\right)} 
             + e^{-d\left(\frac{p}{2} + \frac{\sqrt{2}\tan\alpha}{\sqrt{\tan^{2}\alpha +1}}\right)\left(1+o\left(1\right)\right)}\right), 
\label{estderiv}\eern
for a fixed $C>0$ and for $\epsilon$ sufficiently small.
\label{pr:est}\end{proposition} 
\begin{proof} 
Using the coordinates $y$, we can split 
$u_{\epsilon, Q}\left(y\right)=\bar{u}_{\epsilon, Q}\left(y\right)+\check{u}_{\epsilon, Q}\left(y\right)$, 
where $\bar{u}_{\epsilon, Q}$ is defined in $(\ref{function})$ and 
\be   \check{u}_{\epsilon, Q}\left(y\right) = 
     \chi_{\mu_{0}}\left(\epsilon y\right)\left[\left(\chi_{D}\left(y\right)-1\right) U_{Q}\left(y\right) - \chi_{D}\left(y\right) \Xi_{d}\left(y\right)  +\epsilon \left(\chi_{0}\left(y_{1}-d\right)-1\right) w_{Q}\left(y\right)\right].     \label{funct2}\ee
Then, if we test the gradient of $I_{\epsilon}$ at $u_{\epsilon, Q}$ 
on any function $v\in H^{1}_{D}\left(\Omega_{\epsilon}\right)$, 
we obtain
\bern     
   I'_{\epsilon}\left(u_{\epsilon, Q}\right)\left[v\right]  &=& 
        \int_{\Omega_{\epsilon}}\left(\nabla_{g}u_{\epsilon, Q}\nabla_{g}v 
            + u_{\epsilon, Q}v\right) dy  -  
           \int_{\Omega_{\epsilon}} u^{p}_{\epsilon, Q}v dy    \nonumber\\
     &=&    \int_{\Omega_{\epsilon}}\left(\nabla_{g}\bar{u}_{\epsilon, Q}\nabla_{g}v 
            + \bar{u}_{\epsilon, Q}v\right) dy  -  
           \int_{\Omega_{\epsilon}} \bar{u}^{p}_{\epsilon, Q}v dy  \nonumber\\
      &+& \int_{\Omega_{\epsilon}}\left(\nabla_{g}\check{u}_{\epsilon, Q}\nabla_{g}v 
            + \check{u}_{\epsilon, Q}v\right) dy  -  
           \int_{\Omega_{\epsilon}} \left( \bar{u}^{p}_{\epsilon, Q} - 
           u^{p}_{\epsilon, Q}\right)  v dy    \nonumber\\
     &=&  I'_{\epsilon}\left(\bar{u}_{\epsilon, Q}\right)\left[v\right]  
           +  A_{1}   + A_{2}, 
\label{grad}\eern
where
\be  A_{1} =  \int_{\Omega_{\epsilon}}\left(\nabla_{g}\check{u}_{\epsilon, Q}\nabla_{g}v 
            + \check{u}_{\epsilon, Q}v\right) dy; 
            \qquad 
     A_{2} =  \int_{\Omega_{\epsilon}} \left( \bar{u}^{p}_{\epsilon, Q} - 
           u^{p}_{\epsilon, Q}\right)  v dy.  \nonumber\ee
By Proposition $\ref{prop:estNeumann}$ and in particular by $(\ref{estNeumann})$ 
we have that $I'_{\epsilon}\left(\bar{u}_{\epsilon, Q}\right)\left[v\right]$ 
is of order at most $\epsilon^{2}$. 
Hence we only need to estimate $A_{1}$ and $A_{2}$ in the last line of $(\ref{grad})$. 

To estimate $A_{1}$ we divide further 
$\check{u}_{\epsilon, Q}=\check{u}_{\epsilon, Q, 1}+\check{u}_{\epsilon, Q, 2}+\check{u}_{\epsilon, Q, 3}$, where
\be  \check{u}_{\epsilon, Q, 1}\left(y\right) =  
     \chi_{\mu_{0}}\left(\epsilon y\right)\left(\chi_{D}\left(y\right)-1\right) 
       U_{Q}\left(y\right);    \qquad
       \check{u}_{\epsilon, Q, 2}\left(y\right) = 
         \chi_{\mu_{0}}\left(\epsilon y\right) 
          \chi_{D}\left(y\right) \Xi_{d}\left(y\right);   \nonumber\ee
\be  \check{u}_{\epsilon, Q, 3}\left(y\right) = \chi_{\mu_{0}}\left(\epsilon y\right) 
          \epsilon \left(\chi_{0}\left(y_{1}-d\right)-1\right) w_{Q}\left(y\right). 
\nonumber\ee
Then we write $A_{1}=A_{1,1}+A_{1,2}+A_{1,3}$, with
\be   A_{1,i} = \int_{\Omega_{\epsilon}}\left(\nabla_{g}\check{u}_{\epsilon, Q,i}\nabla_{g}v  + \check{u}_{\epsilon, Q,i}v\right) dy,    
    \qquad   i=1,2,3. 
\nonumber\ee
Since $\chi_{D}\left(y\right)$ is identically equal to $1$ for 
$\vert y\vert\leq\frac{dD}{16}$ and since $\chi_{0}\left(y_{1}-d\right)-1=0$ 
for $y_{1}\leq d$, from $(\ref{limU})$ and $(\ref{est w})$ we get 
\be   \vert A_{1,1}\vert \leq e^{-\frac{dD}{16}\left(1+o\left(1\right)\right)} 
         \norm{v}_{H^{1}_{D}\left(\Omega_{\epsilon}\right)};    \qquad 
      \vert A_{1,3}\vert \leq C\epsilon\left(1+\vert d\vert^{K}\right) e^{-d} 
        \norm{v}_{H^{1}_{D}\left(\Omega_{\epsilon}\right)}. 
\label{grad1}\ee
To control $A_{1,2}$ we write that
\bern 
  A_{1,2} &=& \int_{\Omega_{\epsilon}}\left(\nabla_{g}\check{u}_{\epsilon, Q,2}\nabla_{g}v  + \check{u}_{\epsilon, Q,2}v\right) dy  
          = \int_{\Omega_{\epsilon}}\left(g^{ij}\partial_{i}\check{u}_{\epsilon, Q,2}\partial_{j}v  + \check{u}_{\epsilon, Q,2}v\right) dy     \nonumber\\
          &=& \int_{\Omega_{\epsilon}}\left(\nabla\check{u}_{\epsilon, Q,2}\nabla v  + \check{u}_{\epsilon, Q,2}v\right) dy   
           +   \int_{\Omega_{\epsilon}}\left(g^{ij}-\delta^{ij}\right)  \partial_{i}\check{u}_{\epsilon, Q,2}\partial_{j}v   dy.
\nonumber\eern
From the condition $(c)$ in Subsection $2.2$ we have that 
$\vert g^{ij}-\delta^{ij}\vert\leq C\epsilon\vert y\vert$; then
\be   \vert A_{1,2} - \int_{\Omega_{\epsilon}}\left(\nabla\check{u}_{\epsilon, Q,2}\nabla v  + \check{u}_{\epsilon, Q,2}v\right) dy \vert \leq 
     C\epsilon \left(\int_{\Omega_{\epsilon}} \vert y\vert^{2} 
         \vert \nabla \check{u}_{\epsilon, Q,2} \vert^{2} dy\right)^{\frac{1}{2}}
      \norm{v}_{H^{1}_{D}\left(\Omega_{\epsilon}\right)}. \nonumber\ee 
Since the support of $\check{u}_{\epsilon, Q,2}$ is contained in the set 
$\left\lbrace\vert y\vert\leq\frac{dD}{8}\right\rbrace$, 
we obtain from the last formula and $(\ref{Xinorm})$ that
\be   \vert A_{1,2} - \int_{\Omega_{\epsilon}}\left(\nabla\check{u}_{\epsilon, Q,2}\nabla v  + \check{u}_{\epsilon, Q,2}v\right) dy \vert \leq 
     C\epsilon d D  e^{-d\left(1+o\left(1\right)\right)} 
      \norm{v}_{H^{1}_{D}\left(\Omega_{\epsilon}\right)}. \nonumber\ee 
Now, since $\Xi_{d}$ satisfies $(\ref{Xiprob})$, we have 
\bern 
   \int_{\Omega_{\epsilon}}\left(\nabla\check{u}_{\epsilon, Q,2}\nabla v  + \check{u}_{\epsilon, Q,2}v\right) dy  =   \nonumber\\
    \int_{\Omega_{\epsilon}} \left(\nabla\left(\Xi_{d}\left(y\right) 
    \left(\chi_{\mu_{0}}\left(\epsilon y\right)\chi_{D}\left(y\right)-1\right)\right) 
       \nabla v    +   \Xi_{d}\left(y\right) 
       \left(\chi_{\mu_{0}}\left(\epsilon y\right)\chi_{D}\left(y\right)-1\right) v 
       \right) dy. 
\label{grad2}\eern 
Since also $Dd<\frac{1}{C_{\Omega}}\frac{\mu_{0}}{\epsilon}$, 
the function $\chi_{\mu_{0}}\left(\epsilon y\right)\chi_{D}\left(y\right)-1$ is 
identically zero in the set $\left\lbrace\vert y\vert\leq\frac{dD}{16}\right\rbrace$ 
if $C_{\Omega}$ is sufficiently large. 
Then, using $(\ref{Xipoint})$, $(\ref{Xigradpoint})$ and the H\"{o}lder 
inequality, we find that (also for $D$ large) 
\bern
  \vert  \int_{\Omega_{\epsilon}} \left(\nabla\left(\Xi_{d}\left(y\right) 
    \left(\chi_{\mu_{0}}\left(\epsilon y\right)\chi_{D}\left(y\right)-1\right)\right) 
       \nabla v    +   \Xi_{d}\left(y\right) 
       \left(\chi_{\mu_{0}}\left(\epsilon y\right)\chi_{D}\left(y\right)-1\right) v 
       \right) dy \vert   \nonumber\\
       \leq  e^{-\left[\frac{dD}{16}+ \frac{d}{2} \sqrt{\frac{D\tan\alpha\left(\tan\alpha+1\right)}{\tan^{2}\alpha +1}} + 
       \frac{2d\tan\alpha}{\sqrt{\tan^{2}\alpha}+1}\right] \left(1+o\left(1\right)\right)} 
      \norm{v}_{H^{1}_{D}\left(\Omega_{\epsilon}\right)}. \label{grad22}\eern 
The last three formulas imply 
\be   \vert A_{1,2}\vert \leq  
     C\left( \epsilon d D  e^{-d\left(1+o\left(1\right)\right)}  + 
        e^{-\left[\frac{dD}{16}+ \frac{d}{2} \sqrt{\frac{D\tan\alpha\left(\tan\alpha+1\right)}{\tan^{2}\alpha +1}} + 
       \frac{2d\tan\alpha}{\sqrt{\tan^{2}\alpha}+1}\right] \left(1+o\left(1\right)\right)} \right) 
       \norm{v}_{H^{1}_{D}\left(\Omega_{\epsilon}\right)}. \nonumber\ee 
From $(\ref{grad1})$ and the latter formula it follows that 
\bern    \vert A_{1} \vert \leq C \left( 
            \epsilon d D  e^{-d\left(1+o\left(1\right)\right)}   + 
            e^{-\left[\frac{dD}{16}+ \frac{d}{2} \sqrt{\frac{D\tan\alpha\left(\tan\alpha+1\right)}{\tan^{2}\alpha +1}} + 
       \frac{2d\tan\alpha}{\sqrt{\tan^{2}\alpha}+1}\right] \left(1+o\left(1\right)\right)}  +  \epsilon \left(1+\vert d\vert^{K}\right) e^{-d} 
       \right) \nonumber\\
        \cdot\norm{v}_{H^{1}_{D}\left(\Omega_{\epsilon}\right)}. \label{grad3}\eern

It remains to estimate $A_{2}$. 
First of all, let us recall that the following inequality holds:
\be    
    \vert \bar{u}^{p}_{\epsilon, Q} - u^{p}_{\epsilon, Q} \vert \leq 
    \left\{ 
\begin{array}{ll} 
   C \vert \bar{u}_{\epsilon, Q}\vert^{p-1} \vert \check{u}_{\epsilon, Q}\vert  \qquad & \mathrm{for\ }  \check{u}_{\epsilon, Q} \in \left(0, \frac{1}{2}\bar{u}_{\epsilon, Q}\right),    \\
    C \vert \bar{u}_{\epsilon, Q}\vert^{p-1} \vert \check{u}_{\epsilon, Q}\vert 
    + C\vert \check{u}_{\epsilon, Q}\vert^{p}   \qquad  & \mathrm{otherwise},
 \end{array} 
\right.         
\label{grad4}\ee
for a fixed constant $C$ depending only on $p$. 
Moreover, using $(\ref{limU})$ and $(\ref{est w})$, we can say that there exists 
a small constant $c_{K,n}$ such that 
\be    \bar{u}_{\epsilon, Q}\left(y\right) \geq \frac{7}{8} 
         \frac{e^{-\vert y\vert}}{1+\vert y\vert^{\frac{n-1}{2}}}; 
         \qquad   \mathrm{for\ } \vert y\vert \leq \frac{1}{\epsilon^{c_{K,n}}}. 
\nonumber\ee
We divide next $\Omega_{\epsilon}$ into the two regions
\be     B_{1} = \left\lbrace \vert y\vert < \min\left\lbrace \frac{d}{2}, \frac{1}{\epsilon^{c_{K,n}}}\right\rbrace \right\rbrace ;    \qquad 
         B_{2} = \Omega_{\epsilon}\setminus B_{1}. \nonumber\ee
For $y\in B_{1}$ we have that $\chi_{\mu_{0}}\left(\epsilon y\right)\equiv 1$, 
$\chi_{D}\left(y\right)\equiv 1$, $\chi_{0}\left(y_{1}-d\right)\equiv 1$, 
and hence $\check{u}_{\epsilon, Q}\left(y\right)\equiv -\Xi_{d}\left(y\right)$. 
By $(\ref{Xipoint})$ we have also that 
$\vert\check{u}_{\epsilon, Q}\left(y\right)\vert = \vert\Xi_{d}\left(y\right)\vert \leq e^{-\frac{d}{2}-\frac{\sqrt{2}d\tan\alpha}{\sqrt{\tan^{2}\alpha +1}}+o\left(d\right)}<\frac{1}{2}\bar{u}_{\epsilon, Q}$ for $y\in B_{1}$. 
This fact, $(\ref{grad4})$ and the H\"{o}lder inequality yield 
\be 
      \int_{B_{1}} \vert \bar{u}^{p}_{\epsilon, Q} - u^{p}_{\epsilon, Q} \vert \vert v\vert dy 
       \leq  C  \int_{B_{1}} \vert \bar{u}_{\epsilon, Q}\vert^{p-1} \vert \check{u}_{\epsilon, Q}\vert \vert v\vert dy 
       \leq C e^{-\frac{d}{2}-\frac{\sqrt{2}d\tan\alpha}{\sqrt{\tan^{2}\alpha +1}}+o\left(d\right)} 
       \norm{v}_{H^{1}_{D}\left(\Omega_{\epsilon}\right)}. \nonumber\ee
On the other hand , in $B_{2}$ we have that 
$\vert\bar{u}^{p}_{\epsilon, Q}\vert <C\left(e^{-\frac{d}{2}+o\left(d\right)} + 
e^{-\frac{1+o\left(1\right)}{\epsilon^{c_{K,n}}}}\right)$ and that 
$\vert\check{u}_{\epsilon, Q}\vert\leq e^{-d+o\left(d\right)}$; 
therefore $(\ref{grad4})$ and the H\"{o}lder inequality 
imply again 
\be  \int_{B_{2}} \vert \bar{u}^{p}_{\epsilon, Q} - u^{p}_{\epsilon, Q} \vert \vert v\vert dy   \leq 
C \left[  \left( e^{-\frac{\left(p-1\right)d}{2}+o\left(d\right)} + 
e^{-\frac{p-1+o\left(1\right)}{\epsilon^{c_{K,n}}}}\right) e^{-d+o\left(d\right)} 
+  e^{-pd+o\left(d\right)}\right] 
\norm{v}_{H^{1}_{D}\left(\Omega_{\epsilon}\right)}. \nonumber\ee
The last two formulas provide 
\be    \vert A_{2}\vert \leq C \left[ e^{-\frac{dp}{2}-\frac{\sqrt{2}d\tan\alpha}{\sqrt{\tan^{2}\alpha +1}}+o\left(d\right)} +  e^{-pd+o\left(d\right)} + 
\left( e^{-\frac{\left(p-1\right)d}{2}+o\left(d\right)} + 
e^{-\frac{p-1+o\left(1\right)}{\epsilon^{c_{K,n}}}}\right) e^{-d+o\left(d\right)} 
\right] \norm{v}_{H^{1}_{D}\left(\Omega_{\epsilon}\right)}. \label{grad5}\ee
Finally, we obtain the conclusion from $(\ref{estNeumann})$, $(\ref{grad})$, 
$(\ref{grad3})$ and $(\ref{grad5})$. 
\end{proof}

We have next another estimate for the functional $I_{\epsilon}$, 
which allows to say that the condition $ii)$ in Subsection $\ref{perturbation}$ holds true 
for $I_{\epsilon}$ and the manifold of the $u_{\epsilon, Q}$'s.

\begin{proposition} 
Let $\mu_{0}$ be the constant appearing in Subsection $\ref{apprNeumann}$. 
Then there exists another constant $C_{\Omega}$, independent of $\epsilon$, 
such that, for $C_{\Omega}\leq d\leq\frac{1}{\epsilon C_{\Omega}}$ and for $Dd<\frac{\mu_{0}}{\epsilon C_{\Omega}}$, 
the functions $u_{\epsilon, Q}$ satisfy 
\bern     \norm{I''_{\epsilon}\left(u_{\epsilon, Q}\right)\left[q\right]}  &\leq &
 C \left(\epsilon^{2}  + \epsilon e^{-d\left(1+o\left(1\right)\right)}  
        + e^{-d\left[\frac{1}{2}\sqrt{\frac{D\tan\alpha\left(\tan\alpha +1\right)}{\tan^{2}\alpha +1}} + \frac{2\tan\alpha}{\sqrt{\tan^{2}\alpha +1}}\right]    \left(1+o\left(1\right)\right)}\right) \norm{q}  \nonumber\\
             &+& C\left(e^{-\frac{d\left(p+1\right)}{2}\left(1+o\left(1\right)\right)} 
             + e^{-d\left(\frac{p}{2} + \frac{\sqrt{2}\tan\alpha}{\sqrt{\tan^{2}\alpha +1}}\right)\left(1+o\left(1\right)\right)}\right) \norm{q}, 
\label{estderiv2}\eern
for some fixed $C>0$ and for $\epsilon$ sufficiently small. 
In the above formula $q$ represents a vector in $H^{1}_{D}\left(\Omega_{\epsilon}\right)$ 
which is tangent to the manifold of the $u_{\epsilon, Q}$'s (when $Q$ varies).
\label{pr:est2}\end{proposition}  
\begin{proof} 
Since the arguments are quite similar to those in the proof 
of Proposition $\ref{pr:est}$, we will be rather quick. 
Using the fact that $\det\left(g^{ij}\right)=1$ and the first line in $(\ref{grad})$, 
for any given test function $v\in H^{1}_{D}\left(\Omega_{\epsilon}\right)$ 
we can write that 
\be  I'_{\epsilon}\left(u_{\epsilon, Q}\right)\left[v\right]  =  
   \sum_{i,j}  \int_{\R^{n}_{+}} \left( g^{ij} \partial_{i}u_{\epsilon, Q} 
    \partial_{j}v + u_{\epsilon, Q}v\right) dy   - 
     \int_{\R^{n}_{+}}  u^{p}_{\epsilon, Q}v dy. 
\nonumber\ee
We want to differentiate next with respect to the parameter $Q$, 
taking first a variation $q_{T}$ of the point $Q$ for which $d$ stays fixed, 
namely we take the tangential derivative to the level set of the distance $d$ 
to the interface. 
Let us notice that in the above formula the dependence on $Q$ is in the metric 
coefficients $g^{ij}$ and in the function $w_{Q}$ appearing in 
the expression of $u_{\epsilon, Q}$ (see $\ref{apprfunc})$. 
Therefore we obtain 
\bern  
     \frac{\partial}{\partial Q_{T}} 
      I'_{\epsilon}\left(u_{\epsilon, Q}\right)\left[v\right]  &=& 
      I''_{\epsilon}\left(u_{\epsilon, Q}\right) 
           \left[\frac{\partial u_{\epsilon, Q}}{\partial Q_{T}}, v\right]  =
      \sum_{i,j}  \int_{\R^{n}_{+}} \frac{\partial g^{ij}}{\partial Q_{T}} 
           \partial_{i}u_{\epsilon, Q_{T}} \partial_{j}v  dy    \nonumber\\
      &+&  \sum_{i,j}  \int_{\R^{n}_{+}}  \left( g^{ij}  
           \partial_{i}\frac{\partial u_{\epsilon, Q}}{\partial Q_{T}} \partial_{j}v 
           +  \frac{\partial u_{\epsilon, Q}}{\partial Q_{T}} v\right) dy 
           -  p \int_{\R^{n}_{+}} u^{p-1}_{\epsilon, Q} 
             \frac{\partial u_{\epsilon, Q}}{\partial Q_{T}} v dy. 
\label{sec1}\eern
From Remark $\ref{rem:gij}$ $(ii)$ we have that $\frac{\partial g^{ij}}{\partial Q_{T}}$ is of order 
$\epsilon^{2}\vert y\vert$. 
Moreover, computing the expression of $\frac{\partial u_{\epsilon, Q}}{\partial Q_{T}}$ 
we obtain $\frac{\partial u_{\epsilon, Q}}{\partial Q_{T}} = \epsilon\chi_{\mu_{0}}\left( \epsilon y\right)\chi_{0}\left(y_{1}-d\right)\frac{\partial w_{Q}}{\partial Q_{T}} = o\left(\epsilon^{2}\left(1+\vert y\vert^{K}\right)e^{-\vert y\vert}\right)$, see Subsection $2.2$ in \cite{GMMP1}. 
Reasoning as in the proof of Proposition $\ref{pr:est}$ we then have 
\be    \norm{\frac{\partial}{\partial Q_{T}} 
      I'_{\epsilon}\left(u_{\epsilon, Q}\right)\left[v\right]} \leq 
      C \epsilon^{2} \norm{v}_{H^{1}_{D}\left(\Omega_{\epsilon}\right)}  \qquad 
      \mathrm{for\ every\ }  v\in H^{1}_{D}\left(\Omega_{\epsilon}\right). 
\label{sec2}\ee

On the other hand, when we take a variation $q_{d}$ of $Q$ along the gradient of $d$, 
similarly to $(\ref{sec1})$ we get
\bern  
     \frac{\partial}{\partial Q_{d}} 
      I'_{\epsilon}\left(u_{\epsilon, Q}\right)\left[v\right]  &=& 
      I''_{\epsilon}\left(u_{\epsilon, Q}\right) 
           \left[\frac{\partial u_{\epsilon, Q}}{\partial Q_{d}}, v\right]  =
      \sum_{i,j}  \int_{\R^{n}_{+}} \frac{\partial g^{ij}}{\partial Q_{d}} 
           \partial_{i}u_{\epsilon, Q_{d}} \partial_{j}v  dy    \nonumber\\
      &+&  \sum_{i,j}  \int_{\R^{n}_{+}}  \left( g^{ij}  
           \partial_{i}\frac{\partial u_{\epsilon, Q}}{\partial Q_{d}} \partial_{j}v 
           +  \frac{\partial u_{\epsilon, Q}}{\partial Q_{d}} v\right) dy 
           -  p \int_{\R^{n}_{+}} u^{p-1}_{\epsilon, Q} 
             \frac{\partial u_{\epsilon, Q}}{\partial Q_{d}} v dy. 
\label{sec3}\eern 
Concerning the derivatives of $g^{ij}$ with respect to $Q_{d}$ we can argue exactly 
as for $Q_{T}$, to find 
\be    \vert \sum_{i,j}  \int_{\R^{n}_{+}} \frac{\partial g^{ij}}{\partial Q_{d}} 
           \partial_{i}u_{\epsilon, Q_{d}} \partial_{j}v  dy \vert \leq 
           C \epsilon^{2} \norm{v}_{H^{1}_{D}\left(\Omega_{\epsilon}\right)}. 
\nonumber\ee
Now, computing the derivative of $u_{\epsilon, Q}$ with respect to $Q_{d}$ is more 
complicated than the previous case, because 
$\frac{\partial u_{\epsilon, Q}}{\partial Q_{d}}$ has a more involved expression. 
If we assume that the cut-off function $\chi_{D}\left(y\right)$ is 
defined as $\bar{\chi}_{D}\left(\frac{y}{d}\right)$ for some fixed $\bar{\chi}_{D}$, 
we obtain 
\bern    \frac{\partial u_{\epsilon, Q}}{\partial Q_{d}} &=& 
          - \chi_{\mu_{0}} \chi_{D}  \frac{\partial\Xi_{d}}{\partial d}  + 
          \frac{1}{d^{2}} \chi_{\mu_{0}} \left( \Xi_{d} - U_{Q}\right)  y \cdot 
            \nabla \bar{\chi}_{D}\left(\frac{y}{d}\right)    +   
            \epsilon \chi_{\mu_{0}} w_{Q} 
        \frac{\partial \chi_{0}\left( y_{1} - d\right)}{\partial Q_{d}}   \nonumber\\
             &+&  \epsilon  \chi_{\mu_{0}} \chi_{0}\left( y_{1} - d\right) 
             \frac{\partial w_{Q}}{\partial Q_{d}}. 
\label{sec4}\eern
It is easy to see that the last two terms in the right hand side give a 
contribution to $(\ref{sec3})$ of order at most 
$\epsilon e^{d\left(1+o\left(1\right)\right)} \norm{v}_{H^{1}_{D}\left(\Omega_{\epsilon}\right)}$ 
and $\epsilon^{2} e^{d\left(1+o\left(1\right)\right)} \norm{v}_{H^{1}_{D}\left(\Omega_{\epsilon}\right)}$ respectively. 
Concerning the second one, we can use the fact that the support of 
$\nabla\chi_{D}$ is contained in the set 
$\left\lbrace\vert y\vert\geq\frac{dD}{16}\right\rbrace$, 
together with $(\ref{Xipoint})$, $(\ref{Xigradpoint})$ to see that the 
contribution of this term is at most of order 
\be \left( e^{-\left(\frac{dD}{16} + \frac{d}{2}\sqrt{\frac{D\tan\alpha\left(\tan\alpha +1\right)}{\tan^{2}\alpha +1}} + \frac{2d\tan\alpha}{\sqrt{\tan^{2}\alpha +1}} \right)\left(1+o\left(1\right)\right)} + e^{-\frac{dD}{16}\left(1+o\left(1\right)\right)} \right) \norm{v}_{H^{1}_{D}\left(\Omega_{\epsilon}\right)}.  \nonumber\ee

We can then focus on the first term in the right hand side of $(\ref{sec4})$, 
and consider the quantity 
\be    -\sum_{i,j} \int_{\R^{n}_{+}} \left( g^{ij} \partial_{i} 
           \left( \chi_{\mu_{0}}\chi_{D}\frac{\partial\Xi_{d}}{\partial d} \right) 
           \partial_{j}v + \chi_{\mu_{0}}\chi_{D}\frac{\partial\Xi_{d}}{\partial d} v 
           \right) dy     +    p \int_{\R^{n}_{+}} u^{p-1}_{\epsilon, Q} 
           \chi_{\mu_{0}}\chi_{D}\frac{\partial\Xi_{d}}{\partial d} v dy.
\label{sec5}\ee 
Now, using condition $(c)$ in Subsection $\ref{apprNeumann}$ and $(\ref{Xidpoint})$, 
if we substitute the coefficients $g^{ij}$ with the Kronecker symbols 
we find a difference of order 
$\epsilon\left( e^{-d\left(1+o\left(1\right)\right)} + 
e^{-d\left(1+\frac{2\tan\alpha}{\sqrt{\tan^{2}\alpha +1}}\right) \left(1+o\left(1\right)\right)}\right)$. 
Next, since $\Xi_{d}$ satisfies $-\Delta\Xi_{d}+\Xi_{d}=0$, 
when we differentiate with respect to $d$ we get the same equation for 
$\frac{\partial\Xi_{d}}{\partial d}$, so reasoning as for $(\ref{grad2})$,  $(\ref{grad22})$, together with $(\ref{dpoint})$, $(\ref{dgradpoint})$, 
we find 
\bern    \vert \int_{\R^{n}_{+}} \left( \nabla\left( 
       \chi_{\mu_{0}}\chi_{D}\frac{\partial\Xi_{d}}{\partial d} \right) \cdot \nabla v 
       +  \chi_{\mu_{0}}\chi_{D}\frac{\partial\Xi_{d}}{\partial d} v\right) dy \vert  
       \leq C e^{-\left(\frac{dD}{16} + \frac{d}{2}\sqrt{\frac{D\tan\alpha\left(\tan\alpha +1\right)}{\tan^{2}\alpha +1}} + \frac{2d\tan\alpha}{\sqrt{\tan^{2}\alpha +1}} \right)\left(1+o\left(1\right)\right)} \nonumber\\
      \cdot \norm{v}_{H^{1}_{D}\left(\Omega_{\epsilon}\right)}.  \nonumber\eern
It remains to estimate the last term in $(\ref{sec5})$. 
Using $(\ref{Xidpoint}$), $(\ref{dpoint})$ and the exponential decay of 
$u_{\epsilon, Q}$ and reasoning with argument similar to those for $(\ref{grad5})$, 
we find that it is of order 
\be    e^{-d\left(1+o\left(1\right)\right)}  \left( 
       e^{-d\left(\frac{p-2}{2} + \frac{\sqrt{2}\tan\alpha}{\sqrt{\tan^{2}\alpha +1}}\right)}  +   
       e^{-\frac{d\left(p-1\right)}{2}}   + o\left( \epsilon^{2}\right) \right) 
        \norm{v}_{H^{1}_{D}\left(\Omega_{\epsilon}\right)}.  \nonumber\ee 
All the above comments yield that 
\bern    \norm{\frac{\partial}{\partial Q_{d}} 
      I'_{\epsilon}\left(u_{\epsilon, Q}\right)\left[v\right]} &\leq & 
      C \left(\epsilon^{2}  + \epsilon e^{-d\left(1+o\left(1\right)\right)}  
        + e^{-d\left[\frac{1}{2}\sqrt{\frac{D\tan\alpha\left(\tan\alpha +1\right)}{\tan^{2}\alpha +1}} + \frac{2\tan\alpha}{\sqrt{\tan^{2}\alpha +1}}\right]    \left(1+o\left(1\right)\right)}\right) \norm{v}_{H^{1}_{D}\left(\Omega_{\epsilon}\right)}  \nonumber\\
             &+& C\left(e^{-\frac{d\left(p+1\right)}{2}\left(1+o\left(1\right)\right)} 
             + e^{-d\left(\frac{p}{2} + \frac{\sqrt{2}\tan\alpha}{\sqrt{\tan^{2}\alpha +1}}\right)\left(1+o\left(1\right)\right)}\right) 
 \norm{v}_{H^{1}_{D}\left(\Omega_{\epsilon}\right)}.  \label{sec6}\eern
From $(\ref{sec2})$ and $(\ref{sec6})$ we finally obtain the desired conclusion. 
\end{proof}

\subsubsection{Case $\frac{\pi}{2}\leq\alpha\leq\pi$} \label{case 2}
In this subsection we introduce the manifold of approximate solutions in 
the case $\frac{\pi}{2}\leq \alpha\leq\pi$. 
Since the construction is substantially the same as in the previous 
subsection, we will be rather sketchy. 

\medskip

Let us consider the solution of $(\ref{logphi})$, $\Phi^{d}$, 
and the function $\Xi_{d}$ defined in $(\ref{Xi})$. 
Reasoning as at the beginning of the Subsection $\ref{case 1}$, we derive norm 
estimate for $\Xi_{d}$: 
\be     \norm{\Xi_{d}}_{H^{1}\left(d\left(\Sigma_{D}-Q_{0}\right)\right)}\leq e^{-d\left(1+o\left(1\right)\right)}.  \nonumber\ee
Moreover, from Proposition $\ref{prop:asymptotic2}$ we also obtain 
pointwise estimates for $\Xi_{d}$ and its gradient.

Now, using the cut-off functions $(\ref{chiD})$, $(\ref{chi0})$, 
we define, in the new coordinates $y$ introduced in Subsection $\ref{apprNeumann}$, 
the functions
\be      u_{\epsilon, Q}\left(y\right)  :=  \chi_{\mu_{0}} \left(\epsilon y\right)\left[ \left(U_{Q}\left(y\right) - \Xi_{d}\left(y\right)\right) \chi_{D}\left(y\right) + \epsilon w_{Q}\left(y\right) \chi_{0}\left(y_{1}-d\right) \right].     \nonumber \ee 
Following the line of the Subsection $\ref{case 1}$ 
we prove that the $u_{\epsilon, Q}$'s are good approximate solutions to $(\ref{problem1})$ for suitable conditions of $Q$. 
Since the computations in the following proposition are 
the same as in Proposition $3.12$ and Proposition $3.13$ in \cite{GMMP1} 
we will omit the proof. 
\begin{proposition} 
Let $\mu_{0}$ be the constant appearing in Subsection $\ref{apprNeumann}$. 
Then there exists another constant $C_{\Omega}$, independent of $\epsilon$, 
such that, for $C_{\Omega}\leq d\leq\frac{1}{\epsilon C_{\Omega}}$ and for $Dd<\frac{\mu_{0}}{\epsilon C_{\Omega}}$, 
the functions $u_{\epsilon, Q}$ satisfy 
\bern     \norm{I'_{\epsilon}\left(u_{\epsilon, Q}\right)}  \leq  C\left(\epsilon^{2}  + \epsilon e^{-d\left(1+o\left(1\right)\right)}  +
        e^{-\frac{d\left(p+1\right)}{2}\left(1+o\left(1\right)\right)} 
             + e^{-\frac{3}{2}d\left(1+o\left(1\right)\right)}\right), 
\label{estderiv3}\eern
and 
\bern     \norm{I''_{\epsilon}\left(u_{\epsilon, Q}\right)\left[q\right]}  \leq 
 C \left(\epsilon^{2}  + \epsilon\exp^{-d\left(1+o\left(1\right)\right)}  +
        e^{-\frac{d\left(p+1\right)}{2}\left(1+o\left(1\right)\right)} 
             + e^{-\frac{3}{2}d \left(1+o\left(1\right)\right)}\right) \norm{q}, 
\label{estderiv4}\eern
for some fixed $C>0$ and for $\epsilon$ sufficiently small. 
In $(\ref{estderiv4})$ $q$ represents a vector in $H^{1}_{D}\left(\Omega_{\epsilon}\right)$ 
which is tangent to the manifold of the $u_{\epsilon, Q}$'s (when $Q$ varies).
\label{pr:est3}\end{proposition}

\subsubsection{Case $\pi <\alpha <2\pi$}
In this subsection we introduce the manifold of approximate solutions in 
the case $\pi <\alpha <2\pi$. 
Also in this case we will be very quick, since the construction is the same as in the previous 
subsections. 

\medskip

Let us consider the solution of $(\ref{logphi})$, $\Phi^{d}$, 
and the function $\Xi_{d}$ defined in $(\ref{Xi})$. 
Reasoning as at the beginning of the Subsection $\ref{case 1}$, we derive norm 
estimate for $\Xi_{d}$: 
\be     \norm{\Xi_{d}}_{H^{1}\left(d\left(\Sigma_{D}-Q_{0}\right)\right)}\leq e^{-d\left(1+o\left(1\right)\right)}.  \nonumber\ee
Moreover, from Proposition $\ref{prop:asymptotic3}$ we also obtain 
pointwise estimates for $\Xi_{d}$ and its gradient.

Now, using the cut-off functions $(\ref{chiD})$, $(\ref{chi0})$, 
we define, in the new coordinates $y$ introduced in Subsection $\ref{apprNeumann}$, 
the functions
\be      u_{\epsilon, Q}\left(y\right)  :=  \chi_{\mu_{0}}\left(\epsilon y\right)\left[ \left(U_{Q}\left(y\right) - \Xi_{d}\left(y\right)\right) \chi_{D}\left(y\right) +\epsilon w_{Q}\left(y\right) \chi_{0}\left(y_{1}-d\right) \right].     \nonumber\ee 
Following the line of the Subsection $\ref{case 1}$ 
we obtain that the $u_{\epsilon, Q}$'s are 
good approximate solutions to $(\ref{problem1})$ for suitable conditions of $Q$. 
Since the computations in the following proposition are very similar to those 
in Proposition $3.12$ and Proposition $3.13$ in \cite{GMMP1} 
we will omit the proof. 
\begin{proposition} 
Let $\mu_{0}$ be the constant appearing in Subsection $\ref{apprNeumann}$. 
Then there exists another constant $C_{\Omega}$, independent of $\epsilon$, 
such that, for $C_{\Omega}\leq d\leq\frac{1}{\epsilon C_{\Omega}}$ and for $Dd<\frac{\mu_{0}}{\epsilon C_{\Omega}}$, 
the functions $u_{\epsilon, Q}$ satisfy 
\bern     \norm{I'_{\epsilon}\left(u_{\epsilon, Q}\right)}  \leq  C\left(\epsilon^{2}  + \epsilon e^{-d\left(1+o\left(1\right)\right)}  +
        e^{-\frac{d\left(p+1\right)}{2}\left(1+o\left(1\right)\right)} 
             + e^{-\frac{d}{2}\left(1+o\left(1\right)\right)}\right), 
\label{estderiv5}\eern
and 
\bern     \norm{I''_{\epsilon}\left(u_{\epsilon, Q}\right)\left[q\right]}  \leq 
 C \left(\epsilon^{2}  + \epsilon\exp^{-d\left(1+o\left(1\right)\right)}  +
        e^{-\frac{d\left(p+1\right)}{2}\left(1+o\left(1\right)\right)} 
             + e^{-\frac{d}{2} \left(1+o\left(1\right)\right)}\right) \norm{q}, 
\label{estderiv6}\eern
for some fixed $C>0$ and for $\epsilon$ sufficiently small. 
In $(\ref{estderiv6})$ $q$ represents a vector in $H^{1}_{D}\left(\Omega_{\epsilon}\right)$ 
which is tangent to the manifold of the $u_{\epsilon, Q}$'s (when $Q$ varies).
\label{pr:est4}\end{proposition}

\section{Proof of Theorem $\ref{th:solution}$}   \label{proof}

To prove our main Theorem we need to derive an expansion 
in terms of $Q$ and $\epsilon$ of the energy of approximate solutions. 
Then we can apply the abstract theory in Subsection $\ref{perturbation}$ to 
obtain the existence result. 

In the case $\frac{\pi}{2}\leq\alpha\leq\pi$ 
the energy expansions for the approximate solutions $u_{\epsilon, Q}$ are 
the same as in the case $\alpha =\pi$, 
see Proposition $4.1$ and Proposition $4.2$ in \cite{GMMP1}. 
Then also the definition of the critical manifold 
and the study of the reduced functional are the same. 
Therefore for the proof of Theorem $\ref{th:solution}$ in 
the case $\frac{\pi}{2}\leq\alpha\leq\pi$ we refer the reader 
to Section $4$ in \cite{GMMP1}. 

In the case $\pi <\alpha <2\pi$, even if the approximate solutions are different from the previous case, 
the energy expansions turn out to be the same. 
Then also in this case we omit the proof of Theorem $\ref{th:solution}$ 
and refer the reader to Section $4$ in \cite{GMMP1}.

In the case $0<\alpha <\frac{\pi}{2}$ 
the energy expansions are quite different, so we will give the proof in the details. 

From now on we will assume $0<\alpha <\frac{\pi}{2}$.

\subsection{Energy expansions for the approximate solutions $u_{\epsilon, Q}$} \label{sec:energy}

Here we expand $I_{\epsilon}\left(u_{\epsilon, Q}\right)$ in terms of $Q$ and 
$\epsilon$, where $u_{\epsilon, Q}$ is the function defined in $(\ref{apprfunc})$. 

\begin{proposition} 
For $\epsilon\rightarrow 0$ and $d=d\left(Q\right)\rightarrow+\infty$, the following expansion holds
\be  I_{\epsilon}\left(u_{\epsilon, Q}\right) = \tilde{C}_{0}   
         - \tilde{C}_{1}\epsilon H\left(\epsilon Q\right) + 
              e^{-2d\left(1+o\left(1\right)\right)}  +   
              e^{\left( -d- \frac{d\sqrt{2} \tan\alpha}{\sqrt{\tan^{2}\alpha +1}} \right) \left(1+o\left(1\right)\right)}  + 
               o\left(\epsilon^{2}\right),  \label{espansione}\ee
where $\tilde{C}_{0}$ and $\tilde{C}_{1}$ are the constants in Proposition 
$\ref{prop:estNeumann}$.
\label{pr:espansione}\end{proposition}  
\begin{proof} 
As in the proof of Proposition $\ref{pr:est}$, let us write 
$u_{\epsilon, Q}\left(y\right)=\bar{u}_{\epsilon, Q}\left(y\right)+\check{u}_{\epsilon, Q}\left(y\right)$, see $(\ref{function})$ and $(\ref{funct2})$. 
Then, using the coordinates $y$ introduced in Subsection $\ref{apprNeumann}$, we find that 
\bern 
    I_{\epsilon}\left(u_{\epsilon, Q}\right) &=& 
        I_{\epsilon}\left(\bar{u}_{\epsilon, Q}\right)   + 
        \int_{\Omega_{\epsilon}} \left( \nabla_{g}\bar{u}_{\epsilon, Q} 
        \nabla_{g}\check{u}_{\epsilon, Q} + \bar{u}_{\epsilon, Q}\check{u}_{\epsilon, Q} 
        \right) dy    +    \frac{1}{2}  \int_{\Omega_{\epsilon}} \left( 
        \vert\nabla_{g}\check{u}_{\epsilon, Q}\vert^{2}  
        +  \check{u}_{\epsilon, Q}^{2}\right) dy    \nonumber\\
        &+& \frac{1}{p+1} \int_{\Omega_{\epsilon}} \left( 
        \vert \bar{u}_{\epsilon, Q} \vert^{p+1} - 
        \vert u_{\epsilon, Q} \vert^{p+1} \right)  dy 
\label{esp1}\eern 
Using condition $(c)$ in Subsection $\ref{apprNeumann}$ we have that 
\bern 
     \vert \int_{\Omega_{\epsilon}} \left( \nabla_{g}\bar{u}_{\epsilon, Q} 
        \nabla_{g}\check{u}_{\epsilon, Q} + \bar{u}_{\epsilon, Q}\check{u}_{\epsilon, Q} 
        \right) dy   -  \int_{\R^{n}_{+}} \left( \nabla\bar{u}_{\epsilon, Q} 
        \nabla\check{u}_{\epsilon, Q} + \bar{u}_{\epsilon, Q}\check{u}_{\epsilon, Q} 
        \right) dy   \vert \nonumber\\
      \leq   C \epsilon  \int_{\R^{n}_{+}} \vert y\vert   
         \vert \nabla\bar{u}_{\epsilon, Q} \vert   
         \vert  \nabla\check{u}_{\epsilon, Q} \vert dy; 
\label{esp2}\eern
\be
     \vert \int_{\Omega_{\epsilon}} \left( 
     \vert\nabla_{g}\check{u}_{\epsilon, Q}\vert^{2} + \check{u}_{\epsilon, Q}^{2}
        \right) dy   -  \int_{\R^{n}_{+}} \left( 
        \vert \nabla\check{u}_{\epsilon, Q} \vert^{2} + \check{u}_{\epsilon, Q}^{2}  
        \right) dy   \vert 
      \leq   C \epsilon  \int_{\R^{n}_{+}} \vert y\vert   
            \vert  \nabla\check{u}_{\epsilon, Q} \vert^{2} dy. 
\label{esp3}\ee 
Concerning $(\ref{esp2})$, we can divide the domain of integration into 
$B_{\frac{d}{2}}\left(0\right)$ and its complement and use 
$(\ref{limU})$, $(\ref{est w})$, $(\ref{Xinorm})$, $(\ref{Xipoint})$, 
$(\ref{Xigradpoint})$ to find 
\be   C \epsilon  \int_{\R^{n}_{+}} \vert y\vert \vert \nabla\bar{u}_{\epsilon, Q} \vert   
         \vert  \nabla\check{u}_{\epsilon, Q} \vert dy \leq  
         C \epsilon \left( e^{-\frac{3}{2}d\left(1 + o\left(1\right)\right)}  +  
         e^{-d\left(1 + \frac{\sqrt{2}\tan\alpha}{\sqrt{\tan^{2}\alpha +1}} 
         \right)    \left(1 + o\left(1\right)\right)}\right). 
\nonumber\ee
For $(\ref{esp3})$, the same estimates yield 
\be   C \epsilon  \int_{\R^{n}_{+}} \vert y\vert   
            \vert  \nabla\check{u}_{\epsilon, Q} \vert^{2} dy  \leq C \epsilon 
             \left( e^{-2d\left(1 + o\left(1\right)\right)}  +  
         e^{-d\left(1 + \frac{\sqrt{2}\tan\alpha}{\sqrt{\tan^{2}\alpha +1}} 
         \right)    \left(1 + o\left(1\right)\right)}\right). 
\nonumber\ee 
The last two formulas, $(\ref{esp1})$, $(\ref{esp2})$, $(\ref{esp3})$ imply 
\bern 
    I_{\epsilon}\left(u_{\epsilon, Q}\right) = 
        I_{\epsilon}\left(\bar{u}_{\epsilon, Q}\right)   + 
        \int_{\R^{n}_{+}} \left( \nabla\bar{u}_{\epsilon, Q} 
        \nabla\check{u}_{\epsilon, Q} + \bar{u}_{\epsilon, Q}\check{u}_{\epsilon, Q} 
        \right) dy    +    \frac{1}{2}  \int_{\R^{n}_{+}} \left( 
        \vert\nabla\check{u}_{\epsilon, Q}\vert^{2}  
        +  \check{u}_{\epsilon, Q}^{2}\right) dy    \nonumber\\
        + \frac{1}{p+1} \int_{\Omega_{\epsilon}} \left( 
        \vert \bar{u}_{\epsilon, Q} \vert^{p+1} - 
        \vert u_{\epsilon, Q} \vert^{p+1} \right)  dy  + 
        o\left( \epsilon \left( e^{-\frac{3}{2}d\left(1+o\left(1\right)\right)} 
        + e^{-d-\frac{\sqrt{2}d\tan\alpha}{\sqrt{\tan^{2}\alpha +1}}} 
        \right) \right).
\label{esp4}\eern 
Using the same notation as in the proof of Proposition $\ref{pr:est}$, 
we write $\check{u}_{\epsilon, Q}=\check{u}_{\epsilon, Q, 1}+\check{u}_{\epsilon, Q, 2}+\check{u}_{\epsilon, Q, 3}$. 
Formulas $(\ref{limU})$ and $(\ref{est w})$ imply 
\be    \vert \int_{\R^{n}_{+}} \left( \nabla\bar{u}_{\epsilon, Q} 
        \nabla\check{u}_{\epsilon, Q, 1} + 
        \bar{u}_{\epsilon, Q}\check{u}_{\epsilon, Q, 1} \right) dy  \vert \leq 
        C e^{-\frac{dD}{16}\left(1 + o\left(1\right)\right)}; 
\nonumber\ee
\be    \vert \int_{\R^{n}_{+}} \left( \nabla\bar{u}_{\epsilon, Q} 
        \nabla\check{u}_{\epsilon, Q, 3} + 
        \bar{u}_{\epsilon, Q}\check{u}_{\epsilon, Q, 3} \right) dy  \vert \leq 
        C \epsilon e^{-2d\left(1 + o\left(1\right)\right)}, 
\nonumber\ee  
from which we deduce that 
\bern  
     \int_{\R^{n}_{+}} \left( \nabla\bar{u}_{\epsilon, Q} 
        \nabla\check{u}_{\epsilon, Q} + \bar{u}_{\epsilon, Q}\check{u}_{\epsilon, Q} 
        \right) dy  &=&   
         \int_{\R^{n}_{+}} \left( \nabla\bar{u}_{\epsilon, Q} 
        \nabla\check{u}_{\epsilon, Q, 2} + 
        \bar{u}_{\epsilon, Q}\check{u}_{\epsilon, Q, 2} \right) dy  \nonumber\\
        &+&   o\left( e^{-\frac{dD}{16}\left(1 + o\left(1\right)\right)} + 
          \epsilon e^{-2d\left(1 + o\left(1\right)\right)}\right). 
\nonumber\eern
Similar estimates also yield 
\be   
     \int_{\R^{n}_{+}} \left(\vert\nabla\check{u}_{\epsilon, Q}\vert^{2}  
        +  \check{u}_{\epsilon, Q}^{2}\right) dy  = 
        \int_{\R^{n}_{+}} \left(\vert\nabla\check{u}_{\epsilon, Q, 2}\vert^{2}  
        +  \check{u}_{\epsilon, Q, 2}^{2}\right) dy   + 
       o\left( e^{-\frac{dD}{16}-d\left(1 + o\left(1\right)\right)} + 
          \epsilon e^{-2d\left(1 + o\left(1\right)\right)}\right). 
\nonumber\ee
From a straightforward computation one finds that for any function $v$ 
\be 
    \nabla\check{u}_{\epsilon, Q, 2}\nabla v  + \check{u}_{\epsilon, Q}^{2} v = 
    \nabla\Xi_{d}\cdot\nabla \left(\chi_{\mu_{0}}\left(\epsilon\cdot\right) 
    \chi_{D}v \right) + \Xi_{d}\chi_{\mu_{0}}\left(\epsilon\cdot\right) 
    \chi_{D}v   
    +  \nabla \left( \chi_{\mu_{0}}\left(\epsilon\cdot\right)\chi_{D}\right) 
    \left( \Xi_{d}\nabla v - v\nabla\Xi_{d}\right). 
\nonumber\ee
Applying this relation for $v=\bar{u}_{\epsilon, Q}$ and $v=\check{u}_{\epsilon, Q, 2}$ 
respectively, and using $(\ref{limU})$, $(\ref{est w})$, $(\ref{Xinorm})$, 
$(\ref{Xipoint})$ and $(\ref{Xigradpoint})$ we find that 
\bern  
     \int_{\R^{n}_{+}} \left( \nabla\bar{u}_{\epsilon, Q} 
        \nabla\check{u}_{\epsilon, Q, 2} + 
        \bar{u}_{\epsilon, Q}\check{u}_{\epsilon, Q, 2} \right) dy 
        =  \int_{\R^{n}_{+}} \left( \nabla  \left( 
        \chi_{\mu_{0}}\left(\epsilon\cdot\right)\chi_{D}\bar{u}_{\epsilon, Q}\right) 
        \nabla\Xi_{d}  +  \chi_{\mu_{0}}\left(\epsilon\cdot\right) 
        \chi_{D}\bar{u}_{\epsilon, Q} \Xi_{d} \right) dy  \nonumber\\
       + o\left( e^{-\left(d+\frac{\sqrt{2}d\tan\alpha}{\sqrt{\tan^{2\alpha +1}}}\right) \left(1 + o\left(1\right)\right)} + 
        e^{-\frac{3}{2}d\left(1 + o\left(1\right)\right)}  + 
        e^{-\left( \frac{dD}{16} + \frac{d}{2}\sqrt{\frac{D\tan\alpha\left(\tan\alpha +1\right)}{\tan^{2}\alpha +1}} + \frac{2d\tan\alpha}{\sqrt{\tan^{2}\alpha +1}} \right) \left(1 + o\left(1\right)\right)}\right);
\nonumber\eern
\be    \int_{\R^{n}_{+}} \left(\vert\nabla\check{u}_{\epsilon, Q, 2}\vert^{2}  
        +  \check{u}_{\epsilon, Q, 2}^{2}\right) dy   = 
         \int_{\R^{n}_{+}} \left( \vert\nabla \left( 
           \chi_{\mu_{0}}\left(\epsilon\cdot\right) \chi_{D}\Xi_{d}\right) \vert^{2} + 
           \left( \chi_{\mu_{0}}\left(\epsilon\cdot\right) \chi_{D}\Xi_{d}\right)^{2} 
           \right) dy. 
\nonumber\ee
Using now the fact that, by our construction, the function 
$\chi_{\mu_{0}}\left(\epsilon\cdot\right)\chi_{D}u_{\epsilon, Q} = 
\chi_{\mu_{0}}\left(\epsilon\cdot\right)\chi_{D} 
\left(\bar{u}_{\epsilon, Q}+\check{u}_{\epsilon, Q}\right)$ 
vanishes on $d\left(\partial\Sigma_{D}-Q_{0}\right)$, 
from $(\ref{Xiprob})$ we obtain 
\bern   
    \int_{\R^{n}_{+}} \left( \nabla  \left( 
        \chi_{\mu_{0}}\left(\epsilon\cdot\right)\chi_{D}\bar{u}_{\epsilon, Q}\right) 
        \nabla\Xi_{d}  +  \chi_{\mu_{0}}\left(\epsilon\cdot\right) 
        \chi_{D}\bar{u}_{\epsilon, Q} \Xi_{d} \right) dy  \nonumber\\
        + \frac{1}{2}   \int_{\R^{n}_{+}} \left( \vert\nabla \left( 
           \chi_{\mu_{0}}\left(\epsilon\cdot\right) \chi_{D}\Xi_{d}\right) \vert^{2} + 
           \left( \chi_{\mu_{0}}\left(\epsilon\cdot\right) \chi_{D}\Xi_{d}\right)^{2} 
           \right) dy   \nonumber\\
        = \frac{1}{2} \int_{\R^{n}_{+}} \left( \nabla  \left( 
        \chi_{\mu_{0}}\left(\epsilon\cdot\right)\chi_{D}\bar{u}_{\epsilon, Q}\right) 
        \nabla\Xi_{d}  +  \chi_{\mu_{0}}\left(\epsilon\cdot\right) 
        \chi_{D}\bar{u}_{\epsilon, Q} \Xi_{d} \right) dy. 
\nonumber\eern
From $(\ref{esp4})$ and the last eight formulas we find 
\bern 
    I_{\epsilon}\left(u_{\epsilon, Q}\right) &=& 
        I_{\epsilon}\left(\bar{u}_{\epsilon, Q}\right)   + 
     \frac{1}{2}   \int_{\R^{n}_{+}} \left( \nabla\bar{u}_{\epsilon, Q} 
        \nabla\check{u}_{\epsilon, Q} + \bar{u}_{\epsilon, Q}\check{u}_{\epsilon, Q} 
        \right) dy   
        + \frac{1}{p+1} \int_{\Omega_{\epsilon}} \left( 
        \vert \bar{u}_{\epsilon, Q} \vert^{p+1} - 
        \vert u_{\epsilon, Q} \vert^{p+1} \right)  dy  \nonumber\\
        &+&    o\left( e^{-\frac{dD}{16}\left(1 + o\left(1\right)\right)}  + 
        e^{-d-\frac{\sqrt{2}d\tan\alpha}{\sqrt{\tan^{2}\alpha +1}}\left(1 + o\left(1\right)\right)}  +  
        \epsilon \left( e^{-\frac{3}{2}d\left(1+o\left(1\right)\right)} 
        + e^{-d-\frac{\sqrt{2}d\tan\alpha}{\sqrt{\tan^{2}\alpha +1}} \left(1+o\left(1\right)\right)} 
        \right) \right).
\nonumber\eern 
From $(\ref{limU})$, $(\ref{est w})$, $(\ref{estNeumann})$ and $(\ref{Xinorm})$ 
we have that 
\bern 
   \int_{\R^{n}_{+}} \left( \nabla\bar{u}_{\epsilon, Q} 
        \nabla\check{u}_{\epsilon, Q} + \bar{u}_{\epsilon, Q}\check{u}_{\epsilon, Q} 
        \right) dy  &=&  
        I'_{\epsilon}\left(\bar{u}_{\epsilon, Q}\right) 
        \left[\check{u}_{\epsilon, Q}\right]    +  \int_{\Omega_{\epsilon}} 
         \vert \bar{u}_{\epsilon, Q}\vert^{p} \check{u}_{\epsilon, Q} dy \nonumber\\
         &\leq &   C\epsilon^{2} e^{-d\left(1+o\left(1\right)\right)} + 
         \int_{\Omega_{\epsilon}} 
         \vert \bar{u}_{\epsilon, Q}\vert^{p} \check{u}_{\epsilon, Q} dy, 
\nonumber\eern 
and then 
\bern 
    I_{\epsilon}\left(u_{\epsilon, Q}\right) &=& 
        I_{\epsilon}\left(\bar{u}_{\epsilon, Q}\right)   + 
     \frac{1}{2}  \int_{\Omega_{\epsilon}} 
         \vert \bar{u}_{\epsilon, Q}\vert^{p} \check{u}_{\epsilon, Q} dy  
      +\frac{1}{p+1} \int_{\Omega_{\epsilon}} \left( 
        \vert \bar{u}_{\epsilon, Q} \vert^{p+1} - 
        \vert u_{\epsilon, Q} \vert^{p+1} \right)  dy  \nonumber\\
        &+&    o\left( e^{-\frac{dD}{16}\left(1 + o\left(1\right)\right)}  + 
        e^{-d-\frac{\sqrt{2}d\tan\alpha}{\sqrt{\tan^{2}\alpha +1}}\left(1 + o\left(1\right)\right)}  +  
        \epsilon \left( e^{-\frac{3}{2}d\left(1+o\left(1\right)\right)} 
        + e^{-d-\frac{\sqrt{2}d\tan\alpha}{\sqrt{\tan^{2}\alpha +1}} \left(1+o\left(1\right)\right)} 
        \right)  \right) \nonumber\\
         &+&   o\left( \epsilon^{2} e^{-d\left(1+o\left(1\right)\right)} \right).
\label{esp5}\eern  
Using a Taylor expansion we can write that 
\be  
     \vert\bar{u}_{\epsilon, Q}\vert^{p+1} - \vert u_{\epsilon, Q}\vert^{p+1} = 
    \left\{ 
\begin{array}{ll} 
   -\left(p+1\right) \vert\bar{u}_{\epsilon, Q}\vert^{p} \vert \check{u}_{\epsilon, Q}\vert + o\left(\vert\bar{u}_{\epsilon, Q}\vert^{p-1}\check{u}^{2}_{\epsilon, Q}\right) 
   \qquad & \mathrm{for\ }  \check{u}_{\epsilon, Q} \in \left(0, \frac{1}{2}\bar{u}_{\epsilon, Q}\right),    \\
    o\left(\vert\bar{u}_{\epsilon, Q}\vert^{p}\vert\check{u}_{\epsilon, Q}\vert 
    + \vert\check{u}_{\epsilon, Q}\vert^{p+1}\right)  \qquad  & \mathrm{otherwise},
 \end{array} 
\right.         
\label{esp6}\ee 
As for the estimate of $A_{2}$ in $(\ref{grad5})$, 
we divide the domain into the two regions $B_{1}$, $B_{2}$, and deduce that 
\bern 
    \frac{1}{p+1} \int_{\Omega_{\epsilon}} \left( 
   \vert \bar{u}_{\epsilon, Q} \vert^{p+1}-\vert u_{\epsilon, Q} \vert^{p+1} \right)dy 
   =  - \int_{\Omega_{\epsilon}}  \vert\bar{u}_{\epsilon, Q}\vert^{p} 
       \check{u}_{\epsilon, Q} dy  \nonumber\\
    + o\left( e^{-\frac{d\left(p+1\right)}{2} - \frac{2d\sqrt{2}\tan\alpha}{\sqrt{\tan^{2}\alpha +1}} \left(1+o\left(1\right)\right)} + 
    e^{-\frac{d\left(p+2\right)}{2}\left(1+o\left(1\right)\right)} + 
    e^{-d\left(1+o\left(1\right)\right)} e^{-\frac{1}{\epsilon^{c_{K,n}}}} \right). 
\nonumber\eern
Therefore using $(\ref{esp5})$ the energy becomes 
\bern 
    I_{\epsilon}\left(u_{\epsilon, Q}\right) &=& 
        I_{\epsilon}\left(\bar{u}_{\epsilon, Q}\right)   - 
     \frac{1}{2}  \int_{\Omega_{\epsilon}} 
         \vert \bar{u}_{\epsilon, Q}\vert^{p} \check{u}_{\epsilon, Q} dy  
      + o\left( e^{-d-\frac{\sqrt{2}d\tan\alpha}{\sqrt{\tan^{2}\alpha +1}} \left(1+o\left(1\right)\right)} + 
      e^{-\frac{d\left(p+2\right)}{2}\left(1+o\left(1\right)\right)} \right) \nonumber\\
       &+& o\left( \epsilon \left( \exp^{-\frac{3}{2}\left(1+o\left(1\right)\right)} + 
      e^{-d-\frac{\sqrt{2}d\vert\tan\alpha\vert}{\sqrt{\tan^{2}\alpha +1}} \left(1+o\left(1\right)\right)} \right)  + 
      \epsilon^{2} e^{-d\left(1+o\left(1\right)\right)} \right). 
\nonumber\eern
From $(\ref{Xipoint})$, the expression of $\check{u}_{\epsilon, Q}$ 
and the estimates in the same spirit as above one finds that 
\be  \int_{\Omega_{\epsilon}} 
         \vert \bar{u}_{\epsilon, Q}\vert^{p} \check{u}_{\epsilon, Q} dy  = 
      -\left( e^{-2d\left(1+o\left(1\right)\right)}   + 
        e^{-d-\frac{\sqrt{2}d\tan\alpha}{\sqrt{\tan^{2}\alpha}+1}\left(1+o\left(1\right)\right)} \right), 
\nonumber\ee
and hence from Proposition $\ref{prop:estNeumann}$ we finally find 
\bern 
    I_{\epsilon}\left(u_{\epsilon, Q}\right) &=& \tilde{C}_{0} 
      -  \tilde{C}_{1}\epsilon H\left(\epsilon Q\right) + O\left( \epsilon^{2}\right) 
      + e^{-2d\left(1+o\left(1\right)\right)}   + 
      e^{-d-\frac{\sqrt{2}d\tan\alpha}{\sqrt{\tan^{2}\alpha}+1}\left(1+o\left(1\right)\right)}   \nonumber\\
      &+&  o\left( e^{-d-\frac{\sqrt{2}d\tan\alpha}{\sqrt{\tan^{2}\alpha +1}} \left(1+o\left(1\right)\right)} + 
      e^{-\frac{d\left(p+2\right)}{2}\left(1+o\left(1\right)\right)} \right) \nonumber\\
       &+& o\left( \epsilon \left( e^{-\frac{3}{2}\left(1+o\left(1\right)\right)} + 
      e^{-d-\frac{\sqrt{2}d\tan\alpha}{\sqrt{\tan^{2}\alpha +1}} \left(1+o\left(1\right)\right)} \right)  + 
      \epsilon^{2} e^{-d\left(1+o\left(1\right)\right)} \right). 
\label{esp7}\eern
The conclusion follows from the Schwartz inequality. 
\end{proof}

We give also a related result about the computation of the derivative of the energy 
with respect to $Q$. 
Again, we will be rather sketchy in the proof since the arguments are 
quite similar to the previous ones. 
\begin{proposition} 
For $\epsilon\rightarrow 0$ and $d=d\left(Q\right)\rightarrow+\infty$, the following expansions hold
\bern  \frac{\partial}{\partial Q_{T}}I_{\epsilon}\left(u_{\epsilon, Q}\right) &=&    
         - \tilde{C}_{1}\epsilon^{2}\nabla_{T} H\left(\epsilon Q\right) + 
               o\left(\epsilon^{2}\right);  \label{espansione1}\\
 \frac{\partial}{\partial Q_{d}}I_{\epsilon}\left(u_{\epsilon, Q}\right) &=&    
         - \tilde{C}_{1}\epsilon^{2} \nabla_{d}H\left(\epsilon Q\right) -   
              e^{\left( -d- \frac{d\sqrt{2} \tan\alpha}{\sqrt{\tan^{2}\alpha +1}} \right) \left(1+o\left(1\right)\right)}  + 
               o\left(\epsilon^{2}\right), \label{espansione2}\eern            
where $\tilde{C}_{0}$ and $\tilde{C}_{1}$ are the constants in Proposition 
$\ref{prop:estNeumann}$.
\label{pr:espansione deriv}\end{proposition}  
\begin{proof} 
After some elementary calculations, recalling the definition of 
$\bar{u}_{\epsilon, Q}$ in $(\ref{function})$, we can write
\bern 
    I'_{\epsilon}\left(u_{\epsilon, Q}\right) 
       \left[\frac{\partial u_{\epsilon, Q}}{\partial Q}\right]  &=& 
       \frac{\partial}{\partial Q}I_{\epsilon}\left(\bar{u}_{\epsilon, Q}\right) + 
       \int_{\Omega_{\epsilon}} \left( \nabla_{g} \bar{u}_{\epsilon, Q} 
       \nabla_{g}\frac{\partial \check{u}_{\epsilon, Q}}{\partial Q}  + 
       \bar{u}_{\epsilon, Q} \frac{\partial \check{u}_{\epsilon, Q}}{\partial Q} 
       \right) dy  - \int_{\Omega_{\epsilon}} \bar{u}^{p}_{\epsilon, Q} 
        \frac{\partial \check{u}_{\epsilon, Q}}{\partial Q} dy   \nonumber\\
      &+&   \int_{\Omega_{\epsilon}} \left( \nabla_{g} \check{u}_{\epsilon, Q} 
       \nabla_{g}\frac{\partial u_{\epsilon, Q}}{\partial Q}  + 
       \check{u}_{\epsilon, Q} \frac{\partial u_{\epsilon, Q}}{\partial Q} 
       \right) dy  +  \int_{\Omega_{\epsilon}} \left( \bar{u}^{p}_{\epsilon, Q} - 
        u^{p}_{\epsilon, Q} \right)  \frac{\partial u_{\epsilon, Q}}{\partial Q} dy, 
\label{der1}\eern
where $\check{u}_{\epsilon, Q}=u_{\epsilon, Q}-\bar{u}_{\epsilon, Q}$ 
was defined in $(\ref{funct2})$. 
The first term on the right hand side is estimated in 
Proposition $\ref{prop:estNeumann}$. 
The next two, integrating by parts and using Proposition 
$\ref{prop:estNeumann}$, can be estimated in terms of a quantity like 
\be    C\epsilon^{2} \int_{\Omega_{\epsilon}}\left( 1+\vert y\vert^{K}\right) 
      \vert \frac{\partial \check{u}_{\epsilon, Q}}{\partial Q}\vert dy. 
\nonumber\ee
From the same arguments as in the proof of Proposition $\ref{pr:est2}$ 
one deduces that the latter integral is of order 
$\epsilon^{2}\left(e^{-2d\left(1+o\left(1\right)\right)} + 
^{-d-\frac{\sqrt{2}d\tan\alpha}{\sqrt{\tan^{2}\alpha +1}}\left(1+o\left(1\right)\right)}\right)$. 
To control the first integral in the last line of $(\ref{der1})$ 
we can reason as for the estimate of $A_{1,2}$ in the proof of Proposition 
$\ref{pr:est}$ to see that this is of order 
$e^{-d\left(1+o\left(1\right)\right)}\left(\epsilon + 
e^{-d\left(1+o\left(1\right)\right)}\right) 
\norm{\frac{\partial u_{\epsilon, Q}}{\partial Q}}_{H^{1}_{D}\left(\Omega_{\epsilon}\right)}$. 
From the proof of Proposition $\ref{pr:est2}$ one can deduce that 
$\norm{\frac{\partial u_{\epsilon, Q}}{\partial Q}}_{H^{1}_{D}\left(\Omega_{\epsilon}\right)}\leq C\left(\epsilon^{2} + 
e^{-d\left(1+o\left(1\right)\right)}\right)$, 
and hence the integral under interest is controlled by 
$o\left(\epsilon^{2}\right)+e^{-3d\left(1+o\left(1\right)\right)}$. 

Finally, the last term in $(\ref{der1})$ can be estimated using a 
Taylor expansion as for the term $A_{2}$ in the proof of 
Proposition $\ref{pr:est}$, and up to higher order is given by 
\be    p \int_{\R^{n}_{+}} U^{p-1}_{Q}\left(y\right)\check{u}_{\epsilon, Q} 
        \nabla U_{Q}\left(y\right)\cdot q dy, 
\nonumber\ee
where $q$ stands either for the variation of $Q$ in the coordinates $y$. 
If $q$ preserves $d$, the latter integral gives a negligible contribution, 
and we find $(\ref{espansione1})$. 
If instead $q$ is directed toward the gradient of $d$ the above estimates 
(and in particular $(\ref{Xipoint})$) allow to deduce $(\ref{espansione2})$. 
\end{proof}

\subsection{Finite-dimensional reduction and study of the constrained functional}  \label{sec:var}
 
We apply now the abstract setting described in Subsection $\ref{perturbation}$. 
In fact, the following two Lemmas hold. 
\begin{lemma} 
If $C_{\Omega}$ is as in the previous section and if we choose 
\be      Z_{\epsilon} = \left\lbrace u_{\epsilon, Q} : C_{\Omega}<d<\frac{1}{\epsilon C_{\Omega}}\right\rbrace ,    \nonumber\ee
then the properties $i)$, $iii)$ and $iv)$ in Subsection $\ref{perturbation}$ hold true, 
with $\gamma=\min\left\lbrace 1, p-1\right\rbrace$. 
\label{lem:varcritica}\end{lemma} 
\begin{proof} 
It is immediate to prove that $i)$ and $iii)$ hold; 
in particular, the value of $\gamma$ comes from 
the standard properties of Nemitski operators. 
Property $iv)$ can be easily deduced from the fact that the kernel of 
the linearization of $(\ref{prob})$ in the half space is spanned by 
$\frac{\partial U}{\partial x_{1}}, \ldots, \frac{\partial U}{\partial x_{n-1}}$, 
as proved in \cite{Oh}, and from some localization arguments which can be 
found in Subsections $4.2$, $9.2$ and $9.3$ of \cite{AM}.   
\end{proof}

\begin{lemma} 
For any small positive constant $\delta$, if we take
\be      Z_{\epsilon} = \left\lbrace u_{\epsilon, Q} : \left(2-\delta\right)\vert\log\epsilon\vert <d<\frac{1}{\epsilon C_{\Omega}}\right\rbrace ,    \nonumber\ee
then also property $ii)$ in Subsection $\ref{perturbation}$ holds true, 
with 
\be   f\left(\epsilon\right)= \epsilon ^{\min \left\lbrace  3-\delta, 
          \frac{p+1}{2}\left(2-\delta\right), 
\left(2-\delta\right) \left(\frac{1}{2}\sqrt{\frac{D\tan\alpha\left(\tan\alpha +1\right)}{\tan^{2}\alpha +1}} + \frac{2\tan\alpha}{\sqrt{\tan^{2}\alpha +1}}\right), 
\left(2-\delta\right) \left(\frac{p}{2} + \frac{\sqrt{2}\tan\alpha}{\sqrt{\tan^{2}\alpha +1}}\right) 
     \right\rbrace }  .    \nonumber\ee
\label{lem:varcritica2}\end{lemma}  
\begin{proof} 
This lemma simply follows from Propositions $\ref{pr:est}$ and $\ref{pr:est2}$.    
\end{proof}

As a corollary of the above two lemmas we can apply Proposition $\ref{pr:fi}$ 
and Theorem $\ref{th:rid}$, so we expand next the reduced functional and its 
gradient on the natural constraint $\tilde{Z}_{\epsilon}$. 
\begin{proposition} 
With the choice of $\tilde{Z}_{\epsilon}$ in Lemma $\ref{lem:varcritica2}$, 
if $w_{\epsilon}$ is given by Proposition $\ref{pr:fi}$, then we have 
\bern  \mathbf{I}_{\epsilon}\left(u_{\epsilon, Q}\right) := 
       I_{\epsilon}\left(u_{\epsilon, Q} + 
             w_{\epsilon} \left(u_{\epsilon, Q}\right) \right) \nonumber\\
          = \tilde{C}_{0} - \tilde{C}_{1}\epsilon H\left(\epsilon Q\right) + 
       e^{-2d\left(1+o\left(1\right)\right)} + 
       e^{-d\left(1+\frac{\sqrt{2}\tan\alpha}{\sqrt{\tan^{2}\alpha +1}}\right)\left(1+o\left(1\right)\right)} + o\left(\epsilon^{2}\right); 
\label{red}\eern
\bern  \frac{\partial}{\partial Q_{T}}  \mathbf{I}_{\epsilon}\left(u_{\epsilon, Q}\right) 
      &=&   - \tilde{C}_{1}\epsilon^{2}\nabla_{T} H\left(\epsilon Q\right)  + 
      o\left(\epsilon^{2}\right); 
\label{red1}\\
  \frac{\partial}{\partial Q_{d}}  \mathbf{I}_{\epsilon}\left(u_{\epsilon, Q}\right) 
      &=&   - \tilde{C}_{1}\epsilon^{2}\nabla_{d} H\left(\epsilon Q\right)  + 
      e^{-d\left(1+\frac{\sqrt{2}\tan\alpha}{\sqrt{\tan^{2}\alpha +1}}\right)\left(1+o\left(1\right)\right)}  +
      o\left(\epsilon^{2}\right), 
\label{red2}\eern
as $\epsilon\rightarrow 0$, where $\tilde{C}_{0}$ and $\tilde{C}_{1}$ 
are as in Proposition $\ref{pr:espansione}$ 
and where $Q_{T}$, $Q_{d}$ are as in the proof of Proposition $\ref{pr:est2}$. 
\label{pr:reduced}\end{proposition}
\begin{proof} 
By Propositions $\ref{pr:fi}$ and $\ref{pr:est}$ we have that 
\bern 
    \norm{w_{\epsilon}\left(u_{\epsilon, Q}\right)} \leq 
     C_{1} \norm{I'_{\epsilon}\left(u_{\epsilon, Q}\right)} 
     \leq  C \left( \epsilon^{2} + 
     \epsilon e^{-d\left(1+o\left(1\right)\right)} \right) +  \nonumber\\
      C\left( e^{-d\left(\frac{1}{2}\sqrt{\frac{D\tan\alpha\left(\tan\alpha +1\right)}{\tan^{2}\alpha +1}} + \frac{2\tan\alpha}{\sqrt{\tan^{2}\alpha +1}}\right) \left(1+o\left(1\right)\right)}    + 
     e^{-d\left(\frac{p}{2}+ \frac{\sqrt{2}\tan\alpha}{\sqrt{\tan^{2}\alpha +1}}\right) \left(1+o\left(1\right)\right)}  + 
     e^{-\frac{d\left(p+1\right)}{2} \left(1+o\left(1\right)\right)} \right). 
\nonumber\eern 
From the regularity of $I_{\epsilon}$ and Proposition $\ref{pr:espansione}$ 
we then have 
\bern 
    I_{\epsilon}\left(u_{\epsilon, Q}+w_{\epsilon}\left(u_{\epsilon, Q}\right) \right) 
     = I_{\epsilon}\left(u_{\epsilon, Q}\right)  + 
         I'_{\epsilon}\left(u_{\epsilon, Q}\right) 
         \left[w_{\epsilon}\left(u_{\epsilon, Q}\right)\right] + 
         o\left(\norm{w_{\epsilon}\left(u_{\epsilon, Q}\right)}^{2}\right)  \nonumber\\
     =  \tilde{C}_{0} - \tilde{C}_{1}\epsilon H\left(\epsilon Q\right) + 
       e^{-2d\left(1+o\left(1\right)\right)} + 
       e^{-d\left(1+\frac{\sqrt{2}\tan\alpha}{\sqrt{\tan^{2}\alpha +1}}\right)\left(1+o\left(1\right)\right)} + o\left(\epsilon^{2}\right)  \nonumber\\
     + o\left(\epsilon^{6-2\delta} + \epsilon^{\left(p+1\right)\left(2-\delta\right)} 
      +  \epsilon^{\left(2-\delta\right)\left(\sqrt{\frac{D\tan\alpha\left(\tan\alpha +1\right)}{\tan^{2}\alpha +1}} + \frac{4\tan\alpha}{\sqrt{\tan^{2}\alpha +1}}\right)} 
        + \epsilon^{\left(2-\delta\right)\left(p + \frac{2\sqrt{2}\tan\alpha}{\sqrt{\tan^{2}\alpha +1}}\right)}\right). 
\nonumber\eern 
This immediately gives $(\ref{red})$, since $p>1$ and $\delta$ is small. 

The remaining two estimates are also rather immediate for $p\geq 2$ : 
in fact in this case property $iii)$ in Subsection $\ref{perturbation}$ holds true 
for $\gamma =1$, so we also have 
$\norm{\partial_{Q}w_{\epsilon}}\leq Cf\left(\epsilon\right)$ 
by the last statement in Proposition $\ref{pr:fi})$. 
This, together with the Lipschitzianity of $I'_{\epsilon}$ implies that 
\bern 
    \frac{\partial}{\partial Q} \mathbf{I}_{\epsilon}\left(u_{\epsilon, Q}\right) 
     = I'_{\epsilon}\left(u_{\epsilon, Q}+w_{\epsilon}\right) 
         \left[ \partial_{Q}u_{\epsilon, Q} + \partial_{Q}w_{\epsilon}\right] 
     =  \frac{\partial}{\partial Q} I_{\epsilon}\left(u_{\epsilon, Q}\right) \nonumber\\
     +  I''_{\epsilon}\left(u_{\epsilon, Q}\right) 
         \left[w_{\epsilon}, \partial_{Q}u_{\epsilon, Q} \right]  
     +  I''_{\epsilon}\left(u_{\epsilon, Q}\right) 
         \left[w_{\epsilon}, \partial_{Q}w_{\epsilon} \right]  
     +  \norm{w_{\epsilon}}^{\gamma +1} \left(\norm{\partial_{Q}u_{\epsilon, Q}} 
         + \norm{\partial_{Q}w_{\epsilon}}\right) \nonumber\\
     =  \frac{\partial}{\partial Q} I_{\epsilon}\left(u_{\epsilon, Q}\right) + 
        o\left( f\left( \epsilon\right)^{2}\right)  
     =  \frac{\partial}{\partial Q} I_{\epsilon}\left(u_{\epsilon, Q}\right) \nonumber\\
     +  o\left(\epsilon^{6-2\delta} + + \epsilon^{\left(p+1\right)\left(2-\delta\right)} 
      +  \epsilon^{\left(2-\delta\right)\left(\sqrt{\frac{D\tan\alpha\left(\tan\alpha +1\right)}{\tan^{2}\alpha +1}} + \frac{4\tan\alpha}{\sqrt{\tan^{2}\alpha +1}}\right)} 
        + \epsilon^{\left(2-\delta\right)\left(p + \frac{2\sqrt{2}\tan\alpha}{\sqrt{\tan^{2}\alpha +1}}\right)}\right), 
\label{part1}\eern
since $\gamma =1$. 
The last two estimates then follow from Proposition $\ref{pr:espansione deriv}$. 

For the case $1<p<2$, we reason as in the proof of Proposition $4.5$ 
in \cite{GMMP1} to obtain the estimates. 
This concludes the proof. 
\end{proof}

\subsection{Proof of Theorem $\ref{th:solution}$}  \label{sec:proof}

We use degree theory and the previous expansions. 
First of all, since $Q$ is non degenerate for $H\mid_{\Gamma}$, 
we can find a small neighborhood $V$ of $Q$ in $\Gamma$ such that 
$\nabla H\mid_{\Gamma}\neq 0$ on $\partial V$ and such that 
in some set of coordinates 
\be    \deg\left( \nabla H\mid_{\Gamma}, V, 0\right) \neq 0. \nonumber\ee
Then, if $\delta$ is as in Lemma $\ref{lem:varcritica2}$, 
we choose $0<\beta <\frac{\delta}{2}$, and consider the set 
\be   Y = \left\lbrace \left(d,Q\right) : 
          d\in\left(\left(2-\beta\right) \vert\log\epsilon\vert , 
          \left(2+\beta\right) \vert\log\epsilon\vert\right), 
          \epsilon Q\in V \right\rbrace .  
\nonumber\ee
Since $\nabla H\mid_{\Gamma}\left(Q\right)$ corresponds to 
$\nabla_{T}H\left(\epsilon Q\right)$ in the scaled domain $\Omega_{\epsilon}$, 
by using $(\ref{red1})$ and our choice of $V$ we know that, 
as $\epsilon\rightarrow 0$ 
\be    \nabla_{Q_{T}} \mathbf{I}_{\epsilon}\left(u_{\epsilon, Q}\right) 
      =   - \tilde{C}_{1}\epsilon^{2}\nabla_{T} H\left(\epsilon Q\right)  + 
      o\left(\epsilon^{2}\right) \neq 0    
      \qquad \mathrm{on\ } \frac{1}{\epsilon}\partial V. 
\label{th1}\ee
On the other hand, by $(\ref{red2})$ we also have 
\be  \nabla_{Q_{d}} \mathbf{I}_{\epsilon}\left(u_{\epsilon, Q}\right) 
    =  -\epsilon^{\left(2-\beta\right) \left(1+ \frac{\sqrt{2}\tan\alpha}{\sqrt{\tan^{2}\alpha +1}}\right)}    
    \qquad \mathrm{for\ } d=\left(2-\beta\right) \vert\log\epsilon\vert , 
\label{th2}\ee
and 
\be   \nabla_{Q_{d}} \mathbf{I}_{\epsilon}\left(u_{\epsilon, Q}\right)  
   = - \tilde{C}_{1}\epsilon^{2}\nabla_{d} H\left(\epsilon Q\right)  +  
     o\left(\epsilon^{2}\right), 
     \qquad \mathrm{for\ } d=\left(2+\beta\right) \vert\log\epsilon\vert .
\label{th3}\ee 
Since we are assuming that the gradient of $H$ points toward $\partial_{D}\Omega$ 
near the interface $\Gamma$, 
$\nabla_{d} H\left(\epsilon Q\right)$ is negative and therefore the two 
$d$-derivatives in the last two formulas have opposite signs. 
It follows from the product formula for the degree and 
$(\ref{th1})$-$(\ref{th3})$ that 
\be    \deg\left(\nabla\mathbf{I}_{\epsilon}, Y, 0\right)  = 
       -\deg\left(\nabla H\mid_{\Gamma}, V, 0\right) \neq 0, 
\nonumber\ee
which proves the existence of a critical point for $\mathbf{I}_{\epsilon}$ in $Y$. 
Since we can choose $V$ and $\beta$ arbitrarily small, 
the solution has the asymptotic behavior required by the theorem, 
and more precisely by Remark $\ref{rem:th}$ $(b)$: 
the uniqueness of the global maximum follows fro the asymptotics of the solution 
and standard elliptic regularity estimates.

\begin{remark} 
To prove also the assertion in Remark $\ref{rem:th}$ $(a)$, 
using $(\ref{red})$ in the case of local maximum it is easy to construct 
an open set of $Z_{\epsilon}$ where the maximum of $\mathbf{I}_{\epsilon}$ 
at the interior is strictly larger than the maximum at the boundary. 
On the other hand, when we have a local minimum, one can construct a 
mountain-pass path connecting the two points parametrized by 
$\left(\frac{1}{\epsilon}Q, \left(2-\beta\right)\vert\log\epsilon\vert\right)$ and 
$\left(\frac{1}{\epsilon}Q, \left(2+\beta\right)\vert\log\epsilon\vert\right)$. 
Using a suitably truncated pseudo-gradient flow, one can prove that the 
evolution of the path remains inside 
$\frac{1}{\epsilon}V\times\left(\left(2-\beta\right)\vert\log\epsilon\vert, 
\left(2+\beta\right)\vert\log\epsilon\vert\right)$, 
and still find a critical point of $\mathbf{I}_{\epsilon}$. 
\end{remark}

\begin{center}

{\bf Acknowledgements}

\end{center}

\noindent The author heartily thanks \textit{Andrea Malchiodi} for helpful discussions 
and for having proposed her the topic. 
Moreover she wants to thank \textit{Carlo Sinestrari} for his useful suggestions. 
The author has been supported by the project FIRB-Ideas {\em Analysis and Beyond}.

\end{document}